\documentclass[11pt]{amsart}
\usepackage[T1]{fontenc}
\usepackage{amssymb, amsthm, amsmath}

\usepackage[lite]{amsrefs}

\usepackage{amsfonts,amsbsy}

\newtheorem{theorem}{Theorem}[section]

\newtheorem{corollary}[theorem]{Corollary}

\newtheorem{definition}{Definition}[section]
\newtheorem{example}[theorem]{Example}

\newtheorem{lemma}[theorem]{Lemma}

\newtheorem{proposition}[theorem]{Proposition}

\newtheorem{fact}[theorem]{Fact}

\newtheorem{observation}[theorem]{Observation}

\theoremstyle{remark}
\newtheorem{remark}[theorem]{Remark}
\newtheorem{claim}[theorem]{Claim}

\usepackage{bbm}

\let\cc\c
\usepackage{common}

\newcommand{\lset}[3]{\ensuremath{\{ #1_{#2} : #2 < #3\}}}

\DeclareMathOperator{\wt}{wt}\DeclareMathOperator{\pwt}{pwt}

\DeclareMathOperator{\gstpw}{gstpw}\DeclareMathOperator{\gstw}{gstw}
\DeclareMathOperator{\thpw}{{pwt^\tho}}\DeclareMathOperator{\thw}{{wt^\tho}}

\title{Orthogonality and domination in unstable theories}
\author{Alf Onshuus}
\author{Alexander Usvyatsov}

\address{ Alexander Usvyatsov\\ Universidade de Lisboa \\
  Centro de Matem\'{a}tica e Aplica\cc{c}\~{o}es Fundamentais\\
  Av. Prof. Gama Pinto,2\\
  1649-003 Lisboa \\
  Portugal}

\urladdr{http://www.math.ucla.edu/\textasciitilde alexus}

\thanks{The second author was partially supported by FCT grant SFRH / BPD / 34893 /
2007}

\date{\today}

\begin{document}

\newcommand{\bolda}{\m{\bold \a}}
\newcommand{\boldb}{\m{\bold \b}}
\newcommand{\boldaa}{\m{\bold \a \;}}
\newcommand{\boldbb}{\m{\bold \b \;}}

\newcommand{\phvar}[3]{\m{\ph(#1,#2_#3)}}
\newcommand{\psvar}[3]{\m{\psi(#1,#2_#3)}}
\newcommand{\phx}[2]{\phvar{x}{#1}{#2}}
\newcommand{\psx}[2]{\psvar{x}{#1}{#2}}
\newcommand{\phxa}[1]{\phvar{x}{\a}{#1}}
\newcommand{\psxb}[1]{\psvar{x}{\b}{#1}}
\newcommand{\phxai}{\phvar{x}{\a}{i}}
\newcommand{\psxbj}{\psvar{x}{\b}{j}}
\newcommand{\wx}{\m{\mathfrak{X}}}
\newcommand{\wxx}{\m{\mathfrak{X} \;}}

\def\Ind#1#2{#1\setbox0=\hbox{$#1x$}\kern\wd0\hbox to 0pt{\hss$#1\mid$\hss}
\lower.9\ht0\hbox to 0pt{\hss$#1\smile$\hss}\kern\wd0}
\def\ind{\mathop{\mathpalette\Ind{}}}
\def\Notind#1#2{#1\setbox0=\hbox{$#1x$}\kern\wd0\hbox to 0pt{\mathchardef
\nn=12854\hss$#1\nn$\kern1.4\wd0\hss}\hbox to
0pt{\hss$#1\mid$\hss}\lower.9\ht0 \hbox to
0pt{\hss$#1\smile$\hss}\kern\wd0}
\def\nind{\mathop{\mathpalette\Notind{}}}

\def\thind{\mathop{\mathpalette\Ind{}}^{\text{\th}} }
\def\nthind{\mathop{\mathpalette\Notind{}}^{\text{\th}} }
\def\uth{{\rm{\text{U}}}^{\text{\th}} }
\def\tho{\text{\th} }
\def\sb{\langle b_j \rangle }
\def\sa{\langle a_i \rangle }
\def\sbb{\langle \b_j \rangle }
\def\saa{\langle \a_i \rangle }

\begin{abstract}
In the first part of the paper we study orthogonality, domination,
weight, regular and minimal types in the contexts of rosy and
super-rosy theories. Then we try to develop analogous theory for
arbitrary dependent theories.
\end{abstract}

\maketitle

\section{Introduction and preliminaries}

There are several questions that motivated this research. First, it
is natural to extend the concepts of domination, regularity  and
weight to rosy theories (as it has already been done in the simple
unstable context). One reason for doing this is ``coordinatization''
theorems: one would like to analyze an arbitrary type in terms of
types that can be studied and classified more easily: regular (admit
a pregeometry), minimal, etc. We prove several results of this kind
in section 3. These provide a complementary picture to the recent
work of Assaf Hasson and the first author \cite{HaOn2} where minimal
types in super-rosy theories are investigated. For example, the two
articles combined throw some light on types in theories
interpretable in o-minimal structures.

Another motivation came from our desire to understand and develop
the concept of \emph{strong dependence} (\cite{shelah863}). It has
recently become clear
that this notion is strongly connected to weight. In \cite{Us} the
second author shows that every strongly dependent type has
rudimentary finite generically stable weight. Hence a stable theory
is strongly dependent precisely when every type has finite weight.
The latter conclusion has also been observed by Adler in \cite{Ad},
as he studied the notion of ``burden'', which generalizes weight and
makes sense in any theory. A related concept (within the context of
dependent theories) is investigated by the authors in \cite{OnUs}. A
natural question that arises is: given a dependent theory with a
good enough independence relation, does strong dependence always
imply finite ``weight'' ? More precisely, we analyze the following
two questions in this article. Is thorn-weight finite in a strongly
dependent rosy theory? Is there a natural notion of forking weight
in an arbitrary dependent theory, and what is the connection to
strong dependence? We give a positive answer to the first question
in section 2, and address the second one in section 4.

Several directions pursued in this paper require a delicate analysis
of existence of mutually indiscernible (sometimes Morley) sequences.
Claims of this form are proved in section 2 (in rosy context) and
section 4 (for dependent theories). We find these results of
interest on their own and believe that they might have further
applications. One interesting consequence of our analysis which has
several applications beyond the study of weight is that in a
dependent theory dividing (say, over an extension base) can be
always witnessed by a Morley sequence.

\bigskip

The paper is organized as follows:

We start by defining notions related to forking and \th-forking,
quoting some of the relevant results and proving others that will be
needed throughout the paper.


Most of the paper is devoted to understand \th-orthogonality and the
role of \th-regular and minimal types in rosy, super-rosy, and
finite $\uth$-rank structures. We show many results analogous to
those in stable (and simple) theories, and conclude with a strong
decomposition theorem for types of finite rank in rosy theories. As
we have already mentioned, this result suggests that analysis of
minimal types (as is done e.g. in \cite{HaOn2}) leads to
understanding of all types in a rosy theory of finite rank (e.g., a
theory interpretable in an o-minimal structure).

Section \ref{th-orthognonality} gives proofs of certain basic
results on thorn-weight, thorn-domination and regularity. Many of
these proofs follow the lines of classical ones, but we still go
through them carefully, and where the proofs diverge, we give
alternative proofs for the \th-forking context or explain how to
bridge the gaps. In this section we also connect thorn-weight to
strong dependence and show that every type in a strongly dependent
rosy theory has finite thorn-weight.

We have recently learnt that Hans Adler has also written (in an
unpublished note) a proof of the fact that in a rosy theory,
rudimentarily finite \th-weight implies finite \th-weight (Theorem
\ref{thm:finiteweight}). Both his and our proofs of this particular
fact are mostly based on Wagner's argument \cite{wagnerbook} for
simple theories, which is itself a generalization of Hyttinen's
results \cite{Hy} in the stable context.

In contrast, the analysis of finite-rank theories in Section
\ref{finite} is not close to the existing proofs for stable and
simple theories. Several useful technical tools applicable in this
and related contexts are developed, the main one being Proposition
\ref{2}. We believe that these tools should have many applications.

Finally, in Section  \ref{constructing} we finish the paper by
investigating sufficient conditions for existence of mutually
indiscernible sequences in dependent theories and draw certain
conclusions about the meaning of strong dependence, the behavior
of forking and concepts related to weight. In particular it is
shown that dividing in a dependent theory can normally be
witnessed by a Morley sequence.

\subsection{Notations and Assumptions}

Given a theory $T$, we will work inside its monster model denoted by
$\fC$. By ``monster'' we mean that all cardinals we mention are
``small'' (i.e. smaller than saturation of $\fC$), all sets are
small subsets of $\fC$, all models are small elementary submodels of
$\fC$, and truth values of all formulae and all types are calculated
in $\fC$. We denote tuples (finite unless said otherwise) by lower
case letters $a,b,c$ etc, sets by $A,B,C$ etc, models by $M,N$ etc.

By $a\equiv_Ab$ we mean $\tp(a/A) = \tp(b/A)$. Recall that this is
equivalent to having $\sigma \in \Aut(\fC/A)$ satisfying $\sigma(a)
= b$.

Given an order type $O$, a sequence $I = \inseq{a}{i}{O}$ and $j\in
O$, we often denote the set \set{a_i\colon i<j} by $a_{<j}$.
Similarly for $a_{\le j}, a_{>j}$ etc. We also often identify the
sequence $I$ with the set $\cup I$; that is, when no confusion
should arise we write $\tp(a/I)$ etc.

We will write $a \ind_A B$ for ``$\tp(a/AB)$ does not fork over
$A$''
even if $T$ is not
simple. Although non-forking is generally not an independence
relation, we still find this notation convenient.

For simplicity we assume $T = T^{eq}$ for all theories $T$ mentioned
in this paper.

\subsection{\th-forking}
Since a big part of the paper deals with \th-forking and its
properties, we will now define the basic concepts related to this
notion. The following definitions and facts can be found in
\cite{onshuus}.

\begin{definition}
Let $\ph(x,y)$ be a formula, $b$ be a tuple and $C$ be any set. Then
we define the following.

\begin{itemize}
\item $\ph(x,b)$ \emph{strongly divides over $D$} if $b$ is
not algebraic over $D$ and the set

\[ \{\ph(x,b')\}_{b'\models tp(b/D)}\] is $k$-inconsistent for
some $k\in \mathbb{N}$.

\item $\ph(\x,b)$ \emph{\th-divides over $C$} if there is some
$D\supset C$ such that $\ph(x,b)$ strongly divides over $D$.

\item $\ph(x,b)$ \emph{\th-forks over $C$} if there are
finitely many formulas $\psi_1(x,b_1),\dots, \psi_n(x,b_n)$ such
that $\ph(x,b)\vdash \bigvee_i \psi_i(x,b_i)$ and $\psi_i(x,b_i)$
\th-divides over $C$ for $1\leq i\leq n$.
\end{itemize}
\end{definition}

We will define a theory to be \emph{rosy} if it does not admit
infinite \th-forking chains.

\medskip
Naturally, we say that a (partial) type  \th-divides/forks over a
set $A$ if it contains a formula which \th-divides/forks over $A$.

It is convenient to make definition of strong dividing for types
slightly more tricky, since by making sure that the strongly
dividing type uses ``all'' the parameters we are able to use
algebraic closure much more efficiently. We begin with the following
definition in the particular case of a type over a finite set.

\begin{definition}
    Let $p(x,b)$ be a (partial) type over a finite tuple $b$. We say
    that $p(x,b)$ strongly divides over a set $D$ if there is a
    formula $\ph(x,b) \in p(x,b)$ which strongly divides over $D$.
\end{definition}

\begin{remark}\label{strong dividing}
Notice that the definition of strong dividing (for formulas)
implies the following.

\begin{enumerate}
\item If $\ph(x,a)$ strongly divides over $A$, then
for every $b \models \ph(x,a)$ we have $\tp(a/Ab)$ is algebraic
(whereas $\tp(a/A)$ is nonalgebraic).

\item $\ph(x,a)$ strongly divides over $A$ if and only if
\begin{itemize}
\item $a \not\in \acl(A)$
\item For every infinite nonconstant indiscernible sequence $\lseq{a}{i}{\om}$
in $\tp(a/A)$, we have $\set{\ph(x,a_i)\colon i<\om}$ is
inconsistent.
\end{itemize}

\item Let $a,b,A$ be such that $a\not\in acl(A)$. Then $\tp(b/Aa)$
strongly divides over $A$ if and only if for any $b'\models
\tp(b/Aa)$ we have $a\in acl(Ab')$.

\item If $\ph(x,a)$ strongly divides over $A$ and $B\supset A$ is
such that $a\not\in acl(B)$ then $\ph(x,a)$ strongly divides over
$B$.

\end{enumerate}
\end{remark}
\begin{prf}
\begin{enumerate}
\item
    Suppose $\ph(x,a)$ strongly divides over $A$ (so in particular
    $a \not\in\acl(A))$, and let $b \models \ph(x,a)$. By the
    definition, there are only finitely many $a_1, \ldots, a_{k-1}$ (say,
    $a_1 = a$) in $\tp(a/A)$ such that $\ph(b,a_i)$. In particular,
    there are only finitely many realizations of $\tp(a/Ab)$, as
    required.

\item
    The ``only if'' direction is clear. For the ``if'' direction,
    suppose that $\ph(x,a)$ does not strongly divide over $A$, but $a \not\in\acl(A)$. Then
    for every $k<\om$ there is a subset $\set{a_1, \ldots, a_k}$ of $\tp(a/A)$
    such that $\exists x \bigwedge_{i=1}^k\ph(x,a_i)$. By
    compactness, for any cardinal $\mu$ there is a sequence $\lseq{a}{\al}{\mu}$
    of realizations of $\tp(a/A)$ such that $\exists x
    \bigwedge_{i=1}^k\ph(x,a_{\al_i})$ for every $\al_1< \ldots
    <\al_k<\mu$. By Fact \ref{Erdos Rado} there is such an infinite (nonconstant)
    indiscernible sequence.

\item
    The ``only if'' direction follows from (i).

    On the other hand, assume that $a \not \in \acl(A)$,
    $\tp(b/Aa)$ does not strongly divide over $A$, but for any $b'\models
    \tp(b/Aa)$ we have $a\in acl(Ab')$. By (ii), for
    every formula $\ph(x,a) \in \tp(b/Aa)$ there is an indiscernible
    sequence $\lseq{a}{i}{\om}$ in $\tp(a/A)$ such that $\set{\ph(x,a_i)\colon
    i<\om}$ is consistent. Let $p(x,a) = \tp(b/Aa)$. By
    compactness, there is an indiscernible sequence
    $\lseq{a}{i}{\om}$ in $\tp(a/A)$ such that $q(x) = \bigcup_{i<\om}
    p(x,a_i)$ is consistent (and, moreover, $a_0 = a$). Let $b' \models q(x)$. Clearly $a=a_0 \not\in \acl(Ab')$, since $a_i
    \equiv_{Ab'} a_0$ for all $i$. This contradicts the assumptions.

\item follows easily from (ii).


\end{enumerate}
\end{prf}

In view of (iii) above, we define in general

\begin{definition}
    A type $\tp(b/B)$ strongly divides over $A$ if $B$ is
    nonalgebraic over $A$, but is algebraic over $Ab'$ for every $b'
    \models \tp(b/A)$.
\end{definition}

It may be good to point out that the definition of $p(x,a)$ strong
dividing over $A$ is \emph{not} equivalent to having a strong
dividing formula. The main point is that strong dividing is quite
sensitive to the parameters we name, which is not very common in
model theory (it is not closed under elementary equivalence) but
which is quite useful when working with \th-forking.

\medskip

Recall that a formula $\ph(x,y)$ is called \emph{stable} if it does
not have the order property (see \cite{shelahbook}).

\begin{fact}\label{forking NOP}
If a stable formula $\ph(x,y)$ witnesses that a type $p(x,a)$ forks
over $A$, then there is a $\ph$-formula witnessing that $p(x,a)$
\th-forks over $A$.

In particular, in any stable theory the concepts of \th-forking and
forking coincide.
\end{fact}

\begin{proof}
This is Lemma 5.1.1 in \cite{onshuus}.
\end{proof}

As with stable theories, for many of our results we will need the
existence of a global rank based on the independence notion, which
in this case corresponds to \th-forking.

\begin{definition}
Let $M$ be a model. We will define the \emph{$\uth$-rank} to be the
foundation rank of the order given by the \th-forking relation on
types consistent with $M$. A theory $T$ will be called
\emph{super-rosy} whenever the $\uth$-rank of any type in any model
of $T$ is ordinal valued.
\end{definition}

\begin{fact}\label{lascar}
Let $T$ be a super-rosy theory and let $a,b,A$ be subsets of a model
$M$ of $T$.

Then
\[
\uth\left(tp\left(b/aA\right)\right)+\uth\left(tp\left(a/A\right)\right)
<\uth\left(tp\left(ab/A\right)\right)<
\uth\left(tp\left(b/aA\right)\right)\oplus
\uth\left(tp\left(a/A\right)\right).
\]

\end{fact}

\begin{proof}
Theorem 4.1.10 in \cite{onshuus}.
\end{proof}

We will need the following easy but important Observation. It will
allow us to understand how far we need to extend the types to get
\th-dividing from \th-forking and strong dividing from
\th-dividing; it will be key for the proof of the decomposition
theorem for a type of finite \th-rank in Section \ref{finite uth
rank}. The proof is quite close to the proof of Lemmas 3.1, 3.2
and 3.4 in \cite{HaOn}. However, we prove (and need) a slightly
different result, so we include a proof.

\begin{observation}\label{thorn dividing}
\begin{enumerate}

Let $M$ be a model of a rosy theory $T$, and let $a,b,A$ be tuples
(and sets) in $M$. Then the following hold.

\item Let $p(x,a)$ be a type over $Aa$ which \th-forks over $A$.
Then there is a non-\th-forking extension $p(x,a,a')$ of $p(x,a)$
such that $p(x,a')$ \th-divides over $A$.

\item \label{1} Let $a,b$ be elements and $A$ be a set such that
$tp(b/Aa)$
\th-divides over $A$. Then there is some $e$ such that
$b\thind_{Aa} e$ and such that $tp(b/Aea)$ strongly divides over
$Ae$.

In particular, if $\tp(b/Aa)$ is a type of ordinal valued
$\uth$-rank, then $\uth(tp(b/Aae))<\uth(tp(b/Ae))$.
\end{enumerate}

\end{observation}

\begin{proof}
\noindent (i) Let $p(x,a)$ be as in statement (i) of the Observation
and let $b\models p(x,a)$.

By definition, there are finitely many formulas $\ph_i(x,a_i)$ such
that

\[
p(x,a)\vdash \bigvee_{i=1}^n \ph_i(x,a_i)
\]
and $\ph(x,a_i)$ \th-divides over $A$. By extension of
\th-independence we know that there are $a_1',\dots, a_n'\models
\tp(a_1\dots a_n/Aa)$ such that $b\thind_{Aa} a_1'\dots a_n'$.

So $\tp(b/Aaa_1'\dots a_n')\vdash \ph_m(x,a_m')$ for some $m$;
defining $a':=a_m'$ and $p(x,a,a'):=tp(b/Aaa')$, we get by
construction that $p(x,a,a')$ satisfies the statement of the
Observation.

\medskip

\noindent (ii) Let $a,b$ and $A$ be as in statement (ii) of the
Observation. By definition of \th-dividing there is some $e'$ and
some $\ph(x,a)\in tp(b/Aa)$ such that $\ph(x,a)$ strongly divides
over $Ae'$. Note that in particular $a \not\in\acl(Ae')$.

Let $e\models \tp(e'/Aa)$ be such that $b\thind_{Aa} e$. Since
$e\models \tp(e'/Aa)$, strong dividing is preserved. Moreover,
$a\not\in\acl(Ae)$, hence by the definition (alternatively, by
Remark \ref{strong dividing}(v)), $\tp(b/Aae)$ strongly divides over
$Ae$.

\end{proof}

Finally, we will prove the following well known fact.

\begin{fact}\label{fct:Morley}
    Let $a \thind_A B$. Then there is a \th-Morley sequence $I$
    over $B$ based on $A$ starting with $a$.
\end{fact}
\begin{proof}
    First, construct a non-\th-forking sequence $I' = \lseq{a'}{i}{\mu}$
    in $\tp(a/B)$ based
    on $A$ starting with $A$ by the standard construction, that is,
    $a'_0 = a$, $a'_i \equiv_Ba$, $a'_i \thind_A Ba'_{<i}$. Moreover, make $\mu$
    large enough so that using Erd\"os-Rado
    (more precisely, Fact \ref{Erdos Rado}, see also Remark \ref{rmk:indisc}) one can find
    $I$ which is an $\om$-sequence, $B$-indiscernible and
    every  $n$-type of $I$ over $B$ ``appears'' in $I'$. Clearly $I$
    is a \th-Morley sequence over $B$ based on $A$. Moreover, since
    every element of $I'$ satisfies $\tp(a/B)$, so does every
    element of $I$, so by applying an automorphism over $B$ we may
    assume that $I$ starts with $a$.
\end{proof}

\subsection{Dependent theories and generically stable types}

Recall that a theory $T$ is called \emph{dependent} if there does
not exist a formula which exemplifies the independence property. We
are mostly going to use the following equivalent definition:

\begin{fact}\label{fct:depend}
    $T$ is dependent if and only if there do not exist an
    indiscernible sequence $I = \lseq{a}{i}{\lam}$, a formula
    $\ph(x,y)$ and $\b$ such that both
    $$\set{i \colon \models \ph(a_i,b)}$$ and
    $$\set{i \colon \models \neg\ph(a_i,b)}$$
    are unbounded in $\lam$.
\end{fact}

In section 4, we will work with the classical Shelah's notions of
dividing, forking and splitting from \cite{shelahbook}. Definitions
and a quick summary of properties can be found in section 2 of
\cite{Us}. In particular, we will use the following easy (but
important) consequence of dependence (due to Shelah, \cite{Sh783}).


\begin{fact}\label{fact:splitfork} (T dependent)
    Strong splitting implies dividing (and therefore forking).
\end{fact}

We will now define and give the basic properties of generically
stable types.

\begin{definition}\label{dfn:genstab}
    We call a type $p\in \tS(A)$ \emph{generically stable} if
    every
    Morley sequence in it is an indiscernible set.
\end{definition}


The following key properties of generically stable types can be
found in \cite{Us}:

\begin{fact}\label{fact:genstab} (T dependent)
\begin{enumerate}
\item
    $p \in \tS(A)$ is generically stable if and only if some
    Morley sequence in $p$ is an indiscernible set.
\item
    Let $p$ be a generically stable type. If $p \in \tS(A)$, then it is definable over $\acl(A)$. If $p$
    is definable over $A$ (e.g. $A = \acl(A)$), then $p$ is
    stationary (in the sense that it has unique non-forking extensions).
\item
    Let $p$ be generically stable.
    Non-forking defines on the set of realizations of $p$ a stable
    independence relation (that is, a relation
    satisfying all the axioms of a stable independence relation).
\end{enumerate}
\end{fact}

Note that generic stability is not necessarily closed under
extensions. 

\subsection{Strong dependence and dp-minimality}

The following definitions were motivated by the notions of strong
dependence of Shelah (see e.g. \cite{shelah863}) and appear in
\cite{Us} and \cite{OnUs}. In the definitions below we denote tuples
by $\x,\a$ (in order to stress the difference between singletons and
finite tuples of arbitrary length).

\begin{definition}\label{dfn:strongdep}
$\left.\right.$
\begin{enumerate}
\item
    A \emph{randomness pattern} of depth $\ka$ for a (partial) type $p$ over a set $A$ is an
    array $\seq{\b_i^\al \colon \al<\ka, i<\om}$ and formulae $\ph_\al(\x,\y_\al)$ for
    $\al<\ka$ such that
    \begin{enumerate}
    \item
        The sequences $I^\al = \seq{\b^\al_i\colon i<\om}$ are
        mutually indiscernible over $A$, that is, $I^\al$ is
        indiscernible over $AI^{\neq \al}$.
    \item
        $\len(\b^\al_i) = \len(\y_\al)$
    \item
        for every $\eta \in {}^\ka\om$, the set
        $$ \Gamma_\eta = \set{\ph_\al(\x,\b^\al_\eta(\al) \colon \al < \ka} \cup
        \set{\neg\ph_\al(\x,\b^\al_i) \colon \al<\ka, i<\om, i\neq \eta(\al)}$$
        is consistent with $p$.
    \end{enumerate}
\item
    A (partial) type $p$ over a set $A$ is called \emph{strongly dependent} if there do not exist
    formulae $\ph_\al(\x,\y_\al)$ for $\al<\om$ and
    sequences \seq{\b^\al_i \colon i<\om} for $\al<\om$
    mutually indiscernible over $A$ such that
    for every $\eta \in {}^\om\om$, the set
    $$ \Gamma_\eta = \set{\ph_\al(\x,\b^\al_{\eta(\al)} \colon \al < \om} \cup
    \set{\neg\ph_\al(\x,\b^\al_i) \colon \al<\om, i\neq \eta(\al)}$$
    is consistent with $p$.

    In other words, $p$ is called strongly dependent if there does not exist a randomness
    pattern for $p$ of depth $\ka = \om$.
\item
    \emph{Dependence rank} (dp-rk) of a (partial) type $p$ over a set $A$
    is the supremum of all $\ka$ such that there exists a randomness pattern
    for $p$ of depth $\ka$.
    \item
    A (partial) type over a set $A$ is called \emph{dp-minimal} if
    dp-rank of $p$ is 1.

    In other words, $p$ is dp-minimal if there does not exist a randomness pattern
    for $p$ of depth $2$.
\item
    A theory is called strongly dependent/dp-minimal if the partial
    type $x=x$ is (here $x$ is a singleton).
\item
    Let $T$ be dependent. A type $p$ is called \emph{strongly stable} if it is strongly dependent and
    generically stable.
\end{enumerate}
\end{definition}

\begin{remark}
    Note that Shelah basically shows in \cite{shelah863} Observation 1.7 that
    if there exists a type $p(\x)$ which is not strongly dependent, then
    there exists such a type $p'(x)$ with $x$ being a
    \emph{singleton}. Therefore if there exists an non-strongly
    dependent type, then $T$ is not strongly dependent and the
    definitions above make sense.
\end{remark}

Note that if in the definition of a randomness pattern all formulae
are the same, we get the independence property:

\begin{observation}\label{obs:depstrong}
    A theory $T$ is dependent if and only if it does not admit a
    randomness pattern of some/any infinite depth with
    $\ph_\al(\x,\y) = \ph(\x,\y)$ for all $\al$ if and only if $T$ does not admit
    a randomness pattern of depth $|T|^+$.
\end{observation}
\begin{prf}
    By compactness.
\end{prf}

A related notion, which will be convenient for us to consider, was
investigated by Adler in \cite{Ad}. We are going to use a slightly
different terminology (some of it comes from \cite{OnUs}).

\begin{definition}\label{dfn:strong}
$\left.\right.$
\begin{enumerate}
\item
    A \emph{dividing pattern} of depth $\ka$ for a (partial) type $p$ over a set $A$ is an
    array $\seq{\b_i^\al \colon \al<\ka, i<\om}$ and formulae $\ph_\al(\x,\y_\al)$ for
    $\al<\ka$ such that
    \begin{enumerate}
    \item
        The sequences $I^\al = \seq{\b^\al_i\colon i<\om}$ are
        mutually indiscernible over $A$, that is, $I^\al$ is
        indiscernible over $AI^{\neq \al}$.
    \item
        $\len(\b^\al_i) = \len(\y_\al)$
    \item
        for every $\eta \in {}^\ka\om$, the set
        $$ \set{\ph_\al(\x,\b^\al_\eta(\al) \colon \al < \ka} $$
        is consistent with $p$.
    \item
        for every $\al<\ka$ there exists $k_\al<\om$ such that the
        set
        $$ \set{\ph_\al(\x,\b^\al_i) \colon i<\om} $$
        is $k_\al$-inconsistent with $p$.
    \end{enumerate}
    \item
        A (partial) type $p$ over a set $A$ is called \emph{strong}
    if there does not exist a dividing
    pattern for $p$ of depth $\ka = \om$.

\item
    A theory is called strong if every finitary type is strong.
\end{enumerate}
\end{definition}

\begin{remark}\label{rmk:pattern0}
    Note that by mutual indiscernibility in clause (c) of the definition of a
    dividing pattern it is enough to demand that the set
        $$ \set{\ph_\al(\x,\b^\al_0) \colon \al < \ka} $$
        is consistent with $p$.
\end{remark}

The reader is encouraged to have a look in \cite{Ad} for the
discussion of strong theories. A theory is strong and dependent if
and only is it is strongly dependent (as suggested by the name), and
this is mostly the case we are interested in in this article; but
there are also strong theories which are simple unstable, and even
$SOP_2$.

The following easy Lemma was proven by the authors in \cite{OnUs} in
order to establish the connection between randomness and dividing
patterns. It is also implicit in some proofs in \cite{Ad}. We
include the proof for completeness.

\begin{lemma}\label{negation of dividing}
\begin{enumerate}
\item
    Let $p(x)$ be a type over a set $A$, let $I = \langle
    b_i\rangle_{i\in O}$ be a sequence indiscernible over $A$, and let
    $\ph(x,y)$ be a formula such that $p(x)\cup \ph(x,b_i)$ is
    consistent for some (all) $i$ and $\{\ph(x,b_i)\}_{i\in O}$ is
    $k$-inconsistent for some $k\in \mathbb{N}$. Then
    \[
    p(x)\cup \{\ph(x,b_l)\}\cup \{\neg \ph(x,b_i)\}_{i\neq l}
    \]
    is consistent for all $l$.
\item
    Let $p(x)$ be a type over a set $A$, $n<\om$ and let
    $\seq{b_i^\al \colon \al<n, i<\om}$,
    $\set{\ph_\al(x,y_\al)\colon \al<n}$ be a dividing pattern for
    $p$ over $A$ of depth $n$. Then there exists a randomness pattern for $p$
    over $A$ of depth $n$; in fact, the randomness pattern is given by the same
    array and collection of formulae.
\item
    Clause \emph{(ii)} holds also when the depth $n<\om$ is replaced with any
    cardinal $\ka$.
\end{enumerate}
\end{lemma}

\begin{proof}
\begin{enumerate}
\item
    Without loss of generality let us assume that $O = \setQ$ and
    $l=0$. Assume also that $k$ is \emph{minimal} such that the set
    $\Delta = \set{\ph(x,b_i)\colon i\in\setQ}$ is $k$-inconsistent.
    By the assumptions $k>1$.

    By indiscernibility it is enough to show that the set
    $$\set{\ph(x,b_0)}\cup\set{\neg\ph(x,b_i) \colon i\in\setZ,i\neq 0}$$
    is consistent. Since $\Delta$ is $(k-1)$-consistent, the set
    $$\set{\ph(x,b_0)}\cup\set{\ph(x,b_{\frac 1{i+1}}) \colon 1<i<k}$$
    is consistent, realized by some $d$. But $\Delta$ is
    $k$-inconsistent, so clearly
    $$d\models \set{\ph(x,b_0)}\cup\set{\ph(x,b_{\frac 1{i+1}}) \colon 1<i<k}
    \cup\set{\neg\ph(x,b_i) \colon i\in\setZ,i\neq 0}$$
    and we are done.






\item
    A very similar proof (working with $\bigwedge_\al\ph_\al(x,b^\al_0)$ instead of
    $\ph(x,b_0)$) is left to the reader.
\item
    By clause (ii) and compactness.
\end{enumerate}
\end{proof}



\subsection{Morley-Erd\" os-Rado}

We will make use of the following classical result (originally due
to Morley, although often is referred to as ``Erd\"os-Rado
argument'', since it is an easy consequence of Erd\"os-Rado theorem
and compactness):

\begin{fact}\label{Erdos Rado}
    Let $\lam$ be a cardinal. Then there exists $\mu>\lam$ such that
    for every set $A$ of cardinality $\lam$ and a sequence of tuples
    \lseq{a}{i}{\mu} there exists an \om-type $q(x_0,x_1,\cdots)$ of
    an $A$-indiscernible sequence such that for every $n<\om$ there
    exist $i_1<i_2<\ldots<i_{n}<\mu$ such that
    the restriction of $q$ to the first $n$ variables equals
    $\tp(a_{i_1}\ldots a_{i_n}/A)$.

    We will sometimes denote $\mu$ as above by $\mu(\lam)$.
\end{fact}

\begin{remark}\label{rmk:indisc}
    Let $A$ be a set, $A \subseteq B$, $I$ an $A$-indiscernible
    sequence. Then there exists $I'$, $I' \equiv_A I$, $I'$
    indiscernible over $B$.
\end{remark}
\begin{proof}
    First extend $I$ to be long enough so that Fact \ref{Erdos Rado}
    can be applied to it with $\lam = |B|+|T|$.
    Then there exists $I'$ indiscernible over
    $B$ such that every $n$-type of $I'$ over $B$ ``appears'' in $I$. In
    particular $I'$ has the same type over $A$ as $I$ (since $I$ was
    $A$-indiscernible and $A \subseteq B$).
\end{proof}

\section{\th-orthogonality and \th-regularity in rosy
structures.}\label{th-orthognonality}

The first part of this section is devoted to develop the analogue
notions of domination, orthogonality, weight and regularity in the
\th-forking context and the properties such notions have under
different hypothesis. In the mean time we will show the relation
with strong dependence.

Throughout the section we will assume that $T$ is rosy.

\begin{definition}

We define the following.

\begin{itemize}

\item Two types $p(x)$ and $q(x)$ are \emph{weakly \th-orthogonal}
if they are defined over a common domain $B$ and for every tuple
$a\models p'$ and $b\models q'$ we have $a\thind_B b$. This is
denoted by $p \perp^\tho_w q$.

\item Two types $p$ and $q$ are \emph{\th-orthogonal} if every
non-\th-forking extensions $p'$ and $q'$ of $p$ and $q$ respectively
to a common domain are weakly \th-orthogonal. This is denoted by $p
\perp^\tho q$.

\item Let $A$ be a set, $a,b$ tuples. We say that $a$
\emph{\th-dominates} $b$ \emph{over $A$} if for every $c$ we have
$b\nthind_{A} c$ implies $a\nthind_{A} c$. In this case we write $b
\lhd^\tho_A a$.

\item We say that $a,b$ are \emph{th-domination equivalent over $A$}
if they dominate each other over $A$. Clearly, this is an
equivalence relation. In this case we write $a \bowtie^\tho_A b$.

\item Let $p(x)$ and $q(x)$ be types over $A$ and $B$
respectively. We will say that $p(x)$ \emph{\th-dominates} $q(x)$ if
there are $a,b$ realizations of $p$ and $q$ respectively such that
$a\thind_A B$, $b\thind_B A$ and $b \lhd^\tho_{A\cup B} a$. 
If $A=B$, we say that $p$ \th-dominates $q$ \emph{over $A$}.

\item We say that types $p$ and $q$ are \emph{\th-equidominant} if
there are non-forking extensions $p',q'$ of $p,q$ respectively to
a common domain $C$ and realizations $a'\models p', b'\models q'$
which are domination equivalent over $C$.
In this case we write $p \Join^\tho q$.
\end{itemize}

\begin{remark}
    Note that equidominance is not (in general) an equivalence
    relation on types. Note also that if two types dominate each
    other, they are not necessarily equidominant (even if the
    domination is over the same set of parameters $A$), not even in
    stable theories. The problem is that whereas dominance on elements (over a set $A$)
    is transitive, dominance on types is generally not.
    See section 5.2 of \cite{wagnerbook} for further
    discussion of this matter and examples.
\end{remark}

\end{definition}
Now we define \th-pre-weight and \th-weight of a type $p$. We will
denote them by $\thpw(p)$ and $\thw(p)$. Note that Fact
\ref{forking NOP} implies that in stable theories \th-weight
coincides with the usual notion of weight.

\begin{definition}
$\left.\right.$
\begin{itemize}
\item Let $p(x)$ be any type over some set $A$. We will say that
$a, \langle b_i\rangle_{i=1}^n$ witnesses $\thpw(p(x))\geq n$
(\emph{\th-pre-weight of $p$} is at least $n$) if $a\models p(x)$,
$\seq{b_i}_{i=1}^n$ is $A$-\th-independent and $a\nthind_A b_i$ for
all $i,j$. If $n$ is maximal such that such a witness exists, we
will say that $a, \langle b_i\rangle_{i=1}^n$ witnesses
$\thpw(p(x))=n$ and that $p$ \emph{has \th-pre-weight $n$}.

\item
    We say that a type $p$ has \emph{finite}
    \th-pre-weight if $\thpw(p)<\om$.
    We say that a type $p$ has \emph{rudimentary finite}
    \th-pre-weight if one can not find an infinite witness
    $\set{b_i \colon i<\om}$ as in (i) above.

\item Let $p(x)$ be any type over some set $A$. We will say that
$a,B, \langle b_i\rangle_{i=1}^n$ witnesses $\thw(p(x))\geq n$
(\emph{\th-weight of $p$} is at least $n$) if $a\models p(x)$,
$a\thind_A B$, $\seq{b_i}_{i=1}^n$ is $B$-\th-independent and
$a\nthind b_i$ for all $i,j$. If $n$ is maximal such that such a
witness exists, we will say that $a,B, \langle b_i\rangle_{i=1^n}$
witnesses $\thw(p(x))=n$ and that $p$ has \emph{\th-weight $n$}.

\item
    We say that a type $p$ has \emph{finite}
    \th-weight if $\thw(p)<\om$.
    We say that a type $p$ has \emph{rudimentary finite}
    \th-weight if every non-\th-forking extension of $p$ has rudimentary
    finite pre-weight.

\end{itemize}
\end{definition}

It follows from the definition that $\thw(p) \ge n$ if and only if
there exists a non-\th-forking extension of $p$ with
\th-pre-weight at least $n$.

Notice also that one could define infinite \th-pre-weight and
weight as usual, but we will be concerned only with finite
\th-weights in this paper.

\subsection{Finite \th-weight and strong dependence}

Let us first make the obvious connections between \th-weight and
the notion of $cc$-\th-forking studied in \cite{OnUs}.

We recall the definitions. Note that the variable $x$ is a
\emph{singleton}.

\begin{definition}$\left.\right.$\label{cc}
\begin{enumerate}
\item
        We say that a tuple $\langle \ph_i(x,\a^i)\rangle_{i<n}$ and a set $A$ witness
        \emph{$n$-crisscrossed strong-dividing} ($n$-$cc$-strong-dividing) if
    $\models \exists x \bigwedge_i\ph_i(x,\a^i)$,
        $\ph_i(x,\a^i)$ strong divides over $A$ and $\a^i \thind_A \langle \a^j\rangle_{j\neq i}$ for all $i$.

\item
        We say that a tuple $\langle \ph_i(x,\a^i)\rangle_{i< n}$ and a set $A$ witness
        \emph{$n$-crisscrossed \th-dividing} ($n$-$cc$-\th-dividing) if
    $\models \exists x \bigwedge_i\ph_i(x,\a^i)$,
        $\ph_i(x,\a^i)$ \th-divides over $A$ and $\a^i \thind_A \langle \a^j\rangle_{j\neq i}$ for all $i$.

\item
        We say that a tuple $\langle \ph_i(x,\a^i)\rangle_{i< n}$ and a set $A$ witness
        \emph{$n$-crisscrossed \th-forking} ($n$-$cc$-\th-forking) if
    $\models \exists x \bigwedge_i\ph_i(x,\a^i)$,
        $\ph_i(x,\a^i)$ \th-forks over $A$ and $\a^i \thind_A \langle \a^j\rangle_{j\neq i}$ for all $i$.

\item
        We say that $T$ admits $n$-$cc$-\th-forking (or \th-dividing or strong dividing) if there
        exists a tuple $\langle \ph_i(x,\a^i)\rangle_{i< n}$ witnessing
        $n$-$cc$-\th-forking (or \th-dividing or strong forking) over $A$.

 \item
       Let $p$ be a 1-type over a set $A$.
       We say that a tuple $\langle \ph_i(x,\a^i)\rangle_{i< n}$ witnesses
        \emph{$n-cc$ \th-forking}
    (or \th-dividing or strong dividing) \emph{in $p$} if
    $\langle \ph_i(x,\a^i)\rangle_{i< n}$, $A$ witness
        $n$-$cc$-\th-forking (or \th-dividing or strong dividing) and
    the formula $\bigwedge_i\ph_i(x,\a^i)$
    is consistent with $p$.
 \item
        We say that a type $p \in S_1(A)$ admits $n$-$cc$-\th-forking (or \th-dividing or strong dividing) if there
        exists a tuple $\langle \ph_i(x,\a^i)\rangle_{i< n}$ witnessing
        $n$-$cc$-\th-forking (or \th-dividing or strong dividing) in $p$.
\end{enumerate}
\end{definition}

\begin{remark}\label{ccweight}
The following follow straight from the definitions.
\begin{enumerate}
\item
  $T$ admits $n$-$cc$-\th-forking if and only if there exists a set $A$
  and a type $p \in S(A)$ which admits $n$-$cc$-\th-forking.
\item
  Let $T$ be rosy
  Then a type $p \in S_1(A)$ does not admit $n$-$cc$-\th-forking
  if and only if it has pre-\th-weight less than $n$.
\item
  So a rosy
  $T$ does not admit $n$-$cc$-\th-forking if and only if every 1-type has
  pre-\th-weight less than $n$ if and only if every 1-type has \th-weight less than $n$.
\end{enumerate}
\end{remark}

The following is shown in \cite{OnUs}. For completeness we include
the proof in the appendix, see \ref{equivalent}.

\begin{fact}\label{equivalent2}
The following are equivalent for any $p\in S_1(A)$.
\begin{enumerate}
\item $p$ admits $n$-$cc$-\th-forking.

\item $p$ admits $n$-$cc$-\th-dividing.

\item There is an extension $p(x,B)$ of $p(x)$ such that $p(x,B)$
admits $n$-$cc$-strong dividing.

\end{enumerate}
\end{fact}

The main goal of this subsection is to characterize, in rosy
theories, strong dependence in terms of the \th-pre-weight. In order
to do this, we will need to prove existence of mutually \th-Morley
sequences. The procedures will also bring some light as to what is
needed in order to characterize strong dependence within dependent
theories (or Adler's ``strongness'' within arbitrary theories) in
terms of weight with respect to some independence notion.

\begin{observation}\label{extending morley}
Let $\set{I^i\colon i<n}$ be sequences such that $I^i$ is a
\th-Morley sequence over $AI^{<i}$ based on $A$. Then $I^i$ is a
non-\th-forking sequence over $AI^{\neq i}$ based on $A$.

\end{observation}
\begin{proof}
We need to prove that $\a^i_j \thind_A\a^i_{<j}I^{\neq i}$ where we
define $I^i = \lseq{\a^i}{j}{\mu_i}$.

By the assumptions, $\a^i_j \thind_A\a^i_{<j}I^{<i}$ for all $i,j$.
Hence by transitivity and finite character of \th-forking, we have
$I^{>i}\thind_A I^{\le i}$ for all $i$, in particular
$I^{>i}\thind_A \a^i_{\le j}I^{<i}$ for all $i,j$. By transitivity
again, combining $\a^i_j \thind_A\a^i_{<j}I^{<i}$ and
$I^{>i}\thind_A \a^i_{\le j}I^{<i}$, we have $\a^i_{\le j} \thind_A
I^{\neq i}$.

Therefore, since $\a^i_j \thind_A \a^i_{<j}$, we get $\a^i_j
\thind_A\a^i_{<j}I^{\neq i}$, as required.
\end{proof}

\begin{lemma}\label{toca0}
Let $\set{\a^i\colon i<n}$ be a set of tuples and let
$\set{I^i\colon i<n}$ be sequences such that
\begin{itemize}
\item
    For each $i<n$ the sequence $I^i$ is $AI^{<
    i}a^{>i}$-indiscernible.
\item
    $I^i$ starts with $\a^i$.
\end{itemize}
Then there exist sequences $\set{J^i\colon i<n}$ such that
\begin{itemize}
\item
    For each $i<n$ the sequence $J^i$ is $AJ^{\neq
    i}$-indiscernible.
\item
    $I^i \equiv_{Aa^i} J^i$. So in particular, $J^i$ starts with $\a^i$.
\end{itemize}
Moreover, if $I^i$ are \th-Morley sequences over $AI^{<i}a^{>i}$
based on $A$, then we can make $J^i$ Morley over $AI^{\neq i}$ based
on $A$.
\end{lemma}
\begin{proof}
Exactly the same construction is used to prove both parts of the
Lemma. To avoid being repetitive, we will prove the ``moreover''
part. The proof of the first part is the same, except that without
the extra assumptions we cannot get the stronger conclusion. So
assume the $I^i$'s is a \th-Morley sequence over $AI^{<i}a^{>i}$
based on $A$.

We need to make sure that $I^i$ can be made indiscernible over
$AI^{\neq i}$ and not only over $AI^{<i}a^{>i}$. So assume that
$\len(I^i) = \mu_{i} = \mu(\sum\mu_{<i}+|A|+|T|)$ as in Fact
\ref{Erdos Rado}. We will make our way ``backwards'', that is, by
downward induction on $i$, starting with $i = n$.

Assume that for $\ell>i$ we have $I^\ell$ are \th-Morley
$\om$-sequences over $AI^{\neq \ell}$ based on $A$, whereas for
$\ell\le i$ we still have $I^\ell$ of length $\mu_\ell$ which are
\th-Morley sequences over $AI^{<\ell}\a^{>\ell}$ based on $A$,
non-\th-forking over $AI^{\neq \ell}$ (we have the last assumption
by Observation \ref{extending morley}).

By Fact \ref{Erdos Rado} we can find $J^i$ which is an
indiscernible \om-sequence over $AI^{\neq i}$ such that every
$n$-type of $J^i$ over $AI^{\neq i}$ ``appears'' in $I^i$. So in
particular $J^i$ has the same type over $AI^{<i}\a^{>i}$ as $I^i$.
Moreover, since \th-forking has finite character, $J^i$ is
non-\th-forking over $AI^{\neq i}$.

Notice that given a finite tuple $\b$ in $J^i$ the question of
whether for some $\bar \al = \al_1<\ldots <\al_k<\om$ and $\bar
\beta = \beta_1<\ldots<\beta_k < \om$ we have $\a^\ell_{\bar \al}
\equiv_{\b I^{\neq \ell,i}} \a^\ell_{\bar \beta}$ amounts to the
same question over \emph{some} $\b'$ in $I^i$. Since these were
indiscernible, we get that for any $\ell>i$ the sequence $I^\ell$
is still indiscernible over $AI^{\neq \ell,i}J^i$. Using a similar
argument one can also make sure that for $\ell\neq i$ we still
have that $I^\ell$ is a non-\th-forking sequence over $AI^{\neq
\ell,i}J^i$.

\medskip

So the $J^i$ satisfies all the requirements, except for the fact
that we need the first element of it to be $\a^i$. Note, though,
that the first element of $J^i$ has the same type over
$I^{<i}\a^{>i}$ as $\a^i$. So applying an automorphism over
$I^{<i}\a^{>i}$, we obtain a new $J^i$ that starts with $\a^i$ and
new $I^\ell$ for $\ell>i$ which satisfy all the required
properties, completing the proof of the inductive step.


\end{proof}

\begin{lemma}\label{toca} Let $\set{\a^i\colon i<n}$ be a set of tuples
which is \th-independent over a set $A$. Then there exist
sequences $\set{I^i\colon i<n}$ such that
\begin{itemize}
\item
    For each $i<n$ the sequence $I^i$ is a \th-Morley sequence
    over $AI^{\neq i}$ based on $A$. So $I^i$ is $AI^{\neq
    i}$-indiscernible and $I^i \thind_A I^{\neq i}$.
\item
    $I^i$ starts with $\a^i$.
\end{itemize}

\end{lemma}

\begin{proof}

We construct sequences $I^i$ such that $I^i$ is a \th-Morley
sequence over $AI^{<i}\a^{>i}$ based on $A$. By Lemma \ref{toca0},
this is enough in order to obtain the desired conclusion. The
construction is by induction on $i<n$.

\medskip

The case $i=0$ follows from Fact \ref{fct:Morley}.

\medskip


So let $i>0$, and assume that $I^{<i}$ already exist. Note that
$I^0 \thind_A \a^{>0}$, $I^1 \thind_A I^0\a^{>1}$ hence
$I^0I^1\thind_A a^{>1}$. Continuing like this we see that $I^{<i}
\thind_A \a^{\geq i}$. By symmetry and transitivity $\a^i \thind_A
I^{<i}\a^{>i}$, and we can apply Fact \ref{fct:Morley} again.

%

\end{proof}

\bigskip

We are now able to prove that strong dependence implies boundedness
(by $\omega$) of $cc$-strong dividing patterns and of \th-weight.

\begin{proposition}\label{prp:strong div}
If a type $p$ admits a $n$-$cc$-strong-dividing witness then
$dp-rk(p)\geq n$.
\end{proposition}
\begin{proof}
    Let
    $\langle \psi_i(x,\a^i)\rangle_{i<n}$ and a set $A$ witness
        $n$-$cc$-strong dividing, that is,
    $\bigwedge_i\psi_i(x,\a^i)$ is consistent with $p$,
        $\psi_i(x,\a^i)$ strongly divides over $A$ and $\a^i \thind_A \langle \a^j\rangle_{j\neq i}$ for all $i$.

    By the definition of strong dividing $\a^i \not\in\acl(A)$.
    Since $\set{\a^i\colon i<n}$ is \th-independent, we can build
    as in Lemma \ref{toca}  sequences $I^i = \lseq{\a^i}{j}{\om}$
    such that:

    \begin{itemize}
    \item
        $I^i$ is a \th-Morley sequence over $AI^{\neq i}$ based on $A$.
    \item
        $\a^i_0 = \a^i$.
    \end{itemize}

    For each $i<n$ and $k<\om$ denote $\psi^k_i(x) =
    \psi^k_i(x,\a^i_{<k}) =
    \bigwedge_{j<k}\psi_i(x,\a^i_j)$. Note that since
    $\psi_i(x,\a^i_0)$ strongly divides over $A$, for some $k<\om$
    the formula $\psi^k_i(x)$ is inconsistent.

    So we clearly have a \th-dividing pattern (see Definition
    \ref{dfn:strong}) for $p$ of depth $n$; applying Lemma
    \ref{negation of dividing}(ii), we are done.







\end{proof}

\begin{theorem}\label{finite weight}
If $T$ is strongly dependent (and rosy) then every (finitary) type
has rudimentarily finite \th-weight. If $T$ is dp-minimal then every
1-type has \th-weight 1.

Moreover, the conclusion is true if we just assume that $T$ is
strong and rosy.
\end{theorem}

\begin{proof}
    This now follows easily from Remark \ref{ccweight}, Fact \ref{equivalent2} and Proposition \ref{prp:strong div}.

    For the ``moreover'' part note that Lemma \ref{toca} only
    assumes
    rosiness, and in Proposition \ref{prp:strong div} we show, in
    fact, existence of a dividing pattern.
\end{proof}

The rest of the section will be devoted to show the equivalence
between ``rudimentarily finite'' \th-weight and ``finite
\th-weight''.


\subsection{Basic properties}

Here we list the basic properties of \th-weight. Some of the results
and the proofs in this subsection are very similar, and sometimes
completely analogous to the results in simple theories (see section
5.2 of \cite{wagnerbook}).

\begin{lemma}
The following hold.

\begin{enumerate}\label{weight}
\item If $a\thind _A b$ then $\thw(a/A)=\thw(a/Ab)$.

\item $\thw(ab/A)\leq \thw(a/A)+\thw(b/A)$. Equality holds
whenever $a\thind_A b$.
\end{enumerate}
\end{lemma}

\begin{proof}
The proofs are the same as proofs of Lemmas 5.2.3 and 5.2.4 in
\cite{wagnerbook}, replacing instances of forking for \th-forking.
\end{proof}

The following is very easy:

\begin{observation}\label{obs:domextend}
    Suppose that $a$ is \th-dominated by $b$ over a set $A$,
    and $A' \supseteq A$ is such that $a\thind_AA', b\thind_AA'$.
    Then $a$ is \th-dominated by $b$ over $A'$.
\end{observation}

\begin{observation}\label{obs:domweight}
    Suppose $a \lhd^\tho_A b$. Then $\thw(a/A) \le \thw(b/A)$.
\end{observation}
\begin{prf}
    Assume
    that $\thw(a/A) \ge n$. Then there are $A', \lset{c}{i}{n}$
    witnessing this; that is, $a \thind_AA'$, $\lset{c}{i}{n}$ is an
    $A'$-\th-independent set, and $a \nthind_A c_i$ for all $i$.
    Let $b' \equiv_{Aa} b$ be such that $b' \thind_{Aa} A'$. So $ab'
    \thind_A A'$, hence by the previous Observation $a$ is
    dominated by $b'$ over $A'$. So $b' \thind_{A'}
    c_i$ for all $i$. In particular, $\thw(b/A) \ge \thw(b'/A') \ge
    n$, as required.
\end{prf}

\begin{observation}\label{obs:weight1-domination}
\item
\begin{itemize}
\item
    If $p,q \in \tS(A)$ are not \th-weakly orthogonal and $\thpw(q)
    = 1$, then $p$ dominates $q$ over $A$.
\item
    The relation $p \not\perp_A q$ is an equivalence relation on
    types over $A$ of \th-pre-weight 1.
\end{itemize}
\end{observation}
\begin{proof}
    Easy (see 5.2.11 and 5.2.12 in \cite{wagnerbook}).
\end{proof}

%
%

The following two lemmas are easy but very useful.

\begin{lemma}\label{ab-dom-a}
    Assume $b \lhd^\tho_A a$. Then there exists $B$ containing $A$
    such that $a \ind_AB$ (hence $b \ind_AB$) such that $ab
    \lhd^\tho_B a$.
\end{lemma}
\begin{proof}
    We try to choose by induction on $\al<|T|^+$ an increasing and
    continuous sequence of sets $A_\al$ such that $A_0 = A$ and for all
    $\al$ we have:
    \begin{itemize}
    \item
        $ab \nthind_{A_\al} A_{\al+1}$
    \item
        $a \thind_{A_\al}A_{\al+1}$ (hence $b
        \thind_{A_\al}A_{\al+1}$)
    \end{itemize}

    By local character of \th-independence, there is $\al<|T|^+$ such
    that it is impossible to choose $A_{\al+1}$. Denote $B = A_\al$.
    It is easy to see that the all the requirements are satisfied.
\end{proof}

\begin{lemma}\label{bc-dom-a}
    Assume that $ab \lhd^\tho_A a$, $a \nthind_A c$, $b\thind_A c$
    and $\thw(\tp(c/A)) = 1$. Then $bc \lhd^\tho_A a$.
\end{lemma}
    \begin{proof}
        Assume $a \thind_{A'} d$. Since $ab \lhd^\tho_A a$, we have
        $ab \thind_{A}d$, hence $a\thind_{Ab}d$. Let $A'
        = Ab$. Then $c \thind_{A}A'$ and $c \nthind_{A'}
        a$ (otherwise, by transitivity $c \thind_{A} ab$).
        Since $\thw(c/A) = 1$, clearly $c \thind_{A'} d$
        (otherwise, remembering that $a \thind_{A'} d$, we would
        get that $a,d$ witness $\thpw(c/A') \ge 2$). Hence $bc \thind_{A} d$, as required.
    \end{proof}

\begin{observation}\label{weak-domequivalence}
  Let $p\in \tS(B)$ and $a,B,b_1,\ldots,b_n$ witness $\thpw(p) = n$
  Then $a \lhd^\tho_{B} b_1\ldots b_n$.
\end{observation}
\begin{proof}

    Assume $c\nthind_B a$ and $c\thind_B b_1 \ldots b_n$. Then the set
    $\set{c,b_1, \ldots, b_n}$ is $B$-\th-independent, and it witnesses
    $\thpw(a/B) \ge n+1$, a contradiction.
\end{proof}


\subsection{From rudimentarily finite to finite.}

We will now prove that if a type has rudimentarily finite
\th-weight, it has finite \th-weight. As with stable theories, in
order to show this we found it necessary to prove the very
interesting fact that a type of (rudimentarily) finite \th-weight is
\th-equidominant with a finite free product of \th-weight 1 types.

A good start would be showing that every type of rudimentarily
finite weight is ``related'' (in terms of non-\th-orthogonality) to
\th-weight-1 types. The following two lemmas generalize Hyttinen's
results from \cite{Hy} on types in a stable theory, and we adapt his
technique to the rosy context.

\begin{lemma}\label{lem:weight1-domination}
    Let $p \in \tS(A)$, and assume that
  \begin{enumerate}
    \item
      $a,A',\set{b_1,\ldots,b_n}$ witness $\thw(p) \ge n$. That is,
      $a \thind_AA'$, $\set{b_1,\ldots,b_n}$ are \th-independent
      over $A'$ and $a \nthind_{A'} b_i$ for all $i$.
   \item
      There is no $C$ extending $A'$ such that the following
      three conditions hold:
      \begin{enumerate}
    \item
      $a\thind_{A'}C$
    \item
      $b_1b_2\ldots b_{n-1} \thind_{A'}Cb_n$
    \item
      $b_n \nthind_{A'}C$.
      \end{enumerate}
  \end{enumerate}
  Then
  \begin{itemize}
  \item[(1)]
    Whenever $a \thind_{A'} c$ and $a \thind_{A'b_n}c$, we have $b_n \thind_{A'} c$.
  \item[(2)]
    If, furthermore, $\thw(\tp(b_n/A')) > 1$, then there are
    $B$ and $b'_n,b'_{n+1}$ such that $a, B, \set{b_1,\ldots,b_{n-1},b'_n,
    b'_{n+1}}$ witness $\thw(p) \ge n+1$.
  \end{itemize}
\end{lemma}

\begin{prf} $\left.\right.$
    \begin{itemize}
    \item[(1)] Assume $b_n \nthind_{A'} c$ but $a \thind_{A'}c$ and $a \thind_{A'b_n} c$.
    Without loss of generality $c \thind_{A'b_na} b_1 \ldots b_{n-1}$, hence $c \thind_{A'b_n} b_1 \ldots b_{n-1}$. Let $C =
    A'c$. It is easy to see that (a),(b),(c) above hold for $C$
    (e.g. (b) holds by symmetry and transitivity), contradicting
    assumption (ii) of the Lemma.
    \item[(2)]
    Assume $\thw(\tp(b_n/A'))>1$. This means that there are $B \supseteq A'$ and $c,d$
    such that
    \begin{itemize}
    \item
        $b_n \thind_{A'}B$
    \item
        $c\thind_B d$
    \item
        $b_n \nthind_B c$ and $b_n \nthind_B d$
    \end{itemize}
    Without loss of generality $ab_1\ldots b_{n-1} \thind_{A'b_n} Bcd$.
    It is easy to see that the assumptions of the Lemma still hold after replacing
    $A'$ with $B$. So part (1) holds as well.
    In particular, since $b_n \nthind_{B} c$ and $b_n \nthind_{B}
    d$, whereas $a \thind_{Bb_n} c$ and $a \thind_{Bb_n} d$, so
    $a \nthind_{B} c$ and $a \nthind_{B} d$. Choosing $b'_n = c,
    b'_{n+1} = d$, we are done.
    \end{itemize}
\end{prf}

\begin{lemma}\label{weight1}
  Let $p \in \tS(A)$ be a type of rudimentarily finite \th-weight.
  Then $p$ is non-\th-orthogonal to a type of \th-weight 1.

  Moreover, suppose that $a\models p, B = \set{b_i\colon i<m}, d$
  are such that $a,A,\set{b_i\colon i<m}\cup\set{d}$ witness
  $\thw(p) \ge m+1$. Then there exist $D \supseteq A$ and $d'$ such
  that
  \begin{itemize}
  \item
    $\thw(d'/A') = 1$
  \item
    $a,D,\set{b_i\colon i<m}\cup\set{d'}$ witness
    $\thw(p) \ge m+1$.
  \end{itemize}
\end{lemma}
\begin{proof}
    By considering a non-\th-forking extension it is clear that the Lemma follows from the ``moreover'' part.

    We will prove that if the conclusion fails we can witness that $p$ has rudimentary infinite \th-weight,
    thus contradicting the hypothesis of the Lemma.

    Assume towards a contradiction that the conclusion fails and construct
    by induction on $n \ge m$ sets $A_n$ , $B_n$ and tuples $d_n$ such
    that
    \begin{itemize}
    \item
        $B_n = \set{b_i\colon i<n}$, so $|B_n| = n$
    \item
        $A_m = A$, $B_m = B$, $d_m = d$
    \item
        The sequences \lseq{A}{n}{\om} and \lseq{B}{n}{\om} are
        increasing
    \item
        $a,A_n,B_n = \set{b_i\colon i<n}\cup\set{d_n}$ witness $\thw(p) \ge
        n+1$.
    \end{itemize}

    The case $n=m$ is given, so suppose we have $A_n$, $B_n$ and $d_n$ as above.

    By local character of \th-independence, we can
    replace $A$ by $A'$ satisfying the assumptions of Lemma
    \ref{lem:weight1-domination} with $b_n$ there replaced by our $d$: if given some $A'$ there exists a $C$ as in (ii)
    of Lemma \ref{lem:weight1-domination} above, it satisfies
    all the requirements of $A'$ in (i), so we can replace $A'$ with $C$
    and continue; local character of \th-forking and the fact that
    $d \nthind_{A'} C$ guaranties that the process will eventually
    stop. So by Lemma \ref{lem:weight1-domination} (and the assumption towards contradiction),
    we can ``split'' $d$ into two elements $b_n$ and $d_{n+1}$, that is, find $A_{n+1},
    b_n, d_{n+1}$ such that $a,A_{n+1}, \set{b_i\colon i<n}\cup{\set{d_{n+1}}}$ witness $\thw(p) \ge
    n+1$, as required.

    Let $A_\om = \bigcup_{n<\om}A_n$, $B_\om = \bigcup_{n<\om}B_n$.
    Clearly, $B_\om$ is an \emph{infinite witness} for $\thw(p) \ge
    \aleph_0$, contradicting $p$ having rudimentarily finite weight.

\medskip

    Since this construction contradicts our hypothesis, we know that
    for some $n$ we have $\thw(d_n/A_n) = 1$. But
    then $D = A_n, d' = d_n$ satisfy the conditions required in the conclusion of
    the Lemma.
\end{proof}

\medskip

We are finally ready to prove that a type of rudimentarily finite
\th-weight has finite \th-weight. The proof will be based on
Observation \ref{contained in maximal}, but first we make the
following (temporary) definition.

\begin{definition}
Let $p = \tp(a/A)$ be any type.

We will say that a witness $a, A, \lset{b}{i}{m}$ is a \emph{nice
witness of $\thw(p) \ge m$} if $ab_0\ldots b_{m-1} \lhd_A a$ and
$\thw(b_i/A)=1$ for all $i$.

We will say that a witness $a, A, \lset{b}{i}{m}$ of $\thw(p) \ge m$
to $\thw(p)>m$ \emph{is contained} in a witness
$a,A',\lset{b}{i}{n}$ of $\thw(p)\ge n$ if $A\subset A'$ and $m \le
n$. We say that the first witness if \emph{properly contained} in
the second one if $m<n$.

We will say that a (nice) witness is \emph{maximal} if it is not
properly contained in any other (nice) witness.
\end{definition}

\begin{observation}\label{contained in maximal}
Let $p = \tp(a/A)$ be a type of rudimentarily finite weight.

Then every witness $a, A, \lset{b}{i}{m}$ of $\thw(p) \ge m$ is
contained in a \emph{maximal} witness $a,A',\lset{b}{i}{n}$.

Even more, every nice witness $a, A, \lset{b}{i}{m}$ of $\thw(p) \ge
m$ is contained in a witness $a,A',\lset{b}{i}{n}$ to $\thw(p)\ge n$
which is maximal among all nice witnesses.
\end{observation}

\begin{proof} The proof is precisely the same as the proof of Lemma \ref{weight1}
above:

If there is no maximal witness, then we can construct by induction
on $n<\om$ increasing witnesses $A_n, B_n = \lset{B}{i}{n}$;
taking the unions of these sets, get a contradiction.
\end{proof}

Notice that, a priori, this does not mean, that every such maximal
witness has the same size, or that there are no different such
witnesses of finite unbounded cardinalities so that the \th-weight
of $p$ could still be infinite.

The proof of the following lemma shows that the size of any nice
maximal witness (in particular with $\thw(b_i) = 1$) is the same
finite number $n$, which must \emph{a posteriori} be equal to
$\thw(p)$; that every type of rudimentary finite weight has finite
weight follows as an easy corollary.

\begin{lemma}\label{run out of names}
Let $p$ be a type of rudimentarily finite \th-weight. Then any
maximal nice
    witness $a,A',\lset{b}{i}{m}$ of $\thw(p) \ge m$
    satisfies $a \bowtie^\tho_{A'} b_0 \ldots b_{n-1}$.
\end{lemma}

\begin{proof}
Let $a, A'$ and $b_0\ldots b_{n-1}$ be as in the statement of the
lemma. It is clearly enough to make sure that $a \lhd^\tho_{A'}
b_0\ldots b_{n-1}$.

        So suppose $a \nthind_{A'} c$ but $b_0\ldots b_{n-1}\thind_{A'}
        c$. Then by definition $a,A',B\cup\set{c}$ witness $\thw(p)\ge
        n+1$. By Lemma \ref{weight1} there are $D, c'$ such that
        $a,D,B\cup\set{c'}$ witness $\thw(p)\ge n+1$ and  $\thw(c'/D) =
        1$. By Lemma \ref{bc-dom-a}, $b_0 \ldots b_{n-1}c' \lhd^\tho_{A'}
        a$. By Lemma \ref{ab-dom-a} we may assume $ab_0 \ldots b_{n-1}c'
\lhd^\tho_{A'} a$. So $a, A',b_0 \ldots b_{n-1}c'$ is a nice witness
of $\thw(p) \ge n+1$, contradicting the maximality of $a, A',b_0
\ldots b_{n-1}$.
\end{proof}

The following easy observation shows that nice witnesses exist.

\begin{observation}\label{nice witness exists}
    Let $p=\tp(a/A)$ a nonalgebraic type of rudimentarily finite weight. Then there
    exists a nice witness of $\thw(p) \ge 1$.
\end{observation}
\begin{prf}
By Lemma \ref{weight1} we can find $b$, $\thw(b/A')=1$ where $A'$ is
the domain over  which $a \nthind_{A'} b$. Since $\thw(b/A') = 1$
and $a \nthind_{A'} b$, Observation \ref{obs:weight1-domination}
implies that $b \lhd^\tho_A a$. Finally, we can assume $ab
\lhd^\tho_{A'} a$ by Lemma \ref{ab-dom-a}, which finishes the proof.
\end{prf}

We have finally reached our goal.

\begin{theorem}\label{thm:finiteweight}
    Let $p \in \tS(A)$ be a nonalgebraic type of rudimentarily finite \th-weight.
    Then $\thw(p)<\aleph_0$ and $p$ is \th-equidominant
    with a finite free product of \th-weight-1 types.

    More precisely, there
    exist $a,A',\lset{b}{i}{n}$ such that
    \begin{itemize}
    \item
        $a,A',\lset{b}{i}{n}$ witness that $\thw(p) \ge n$
    \item
        $\thw(b_i/A') = 1$ for all $i$
    \item
        $a \bowtie_{A'}^\tho b_0\ldots b_{n-1}$.
    \end{itemize}

\end{theorem}
\begin{prf}

    Let $a, A'$, $B = \lset{b}{i}{n}$ be such that
    \begin{enumerate}
    \item
        $a,A',\lset{b}{i}{n}$ witness that $\thw(p) \ge n$
    \item
        $\thw(b_i/A') = 1$ for all $i$
    \item
        $aB \lhd^\tho_{A'} a$
    \item
        $\lset{b}{i}{n}$ is maximal satisfying (i),(ii) and (iii). In
        other words, if there are $A'' \supseteq A'$, $B'' \supseteq
        B$ satisfying (i), (ii) and (iii), then $B'' = B$.
    \end{enumerate}

    In other words, $a, A', B$ is a maximal nice witness for $\thw(p) \ge n$.
    It is easy to see that such $A', B$ exist: Observation
    \ref{nice witness exists} gives us
a nonempty $B_0$ satisfying (i),(ii) and (iii). Since $p$ has
rudimentary finite weight, by Observation \ref{contained in maximal}
we know that $B_0$ is contained in a maximal $B$, as required in (i)
-- (iv) above.

By Lemma \ref{run out of names},
$a \bowtie^\tho_{A'} b_0
\ldots b_{n-1}$.
By Lemma \ref{weight} and Observation
\ref{obs:domweight} it follows that $p$ has finite weight
$n$.\end{prf}

Reading carefully the proof of the Theorem, we obtain the following
more precise statement.

\begin{corollary}
    Let $p$ be a type of rudimentarily finite \th-weight. Then $\thw(p)=n$ for some $n<\om$, and any maximal
    nice witness $a,A',\lset{b}{i}{m}$ of $\thw(p) \ge m$
        satisfies $m=n$ and $a \bowtie^\tho_{A'} b_0 \ldots b_{n-1}$.
\end{corollary}

\begin{corollary}
    In a strongly dependent (and even strong) rosy theory, every
    type has finite \th-weight.
\end{corollary}
\begin{prf}
    By Theorem \ref{finite weight} and Theorem
    \ref{thm:finiteweight}.
\end{prf}

\subsection{\th-regular types.}

We will finish this section by understanding some implications of
the above results to \th-regular types. The definition is the
analogue of the definition of regular types in the stable and simple
context.

\begin{definition}
A type $r(x)$ over $A$ is \emph{\th-regular} if for any $B\supset
A$ then given any \th-forking extension $q(x)$ of $r(x)$ and a
non-\th-forking extension $p(x)$ of $r(x)$ if $r,q$ are over $B$
then $q(x)$ is weakly \th-orthogonal to $r(x)$.
\end{definition}

The following desired property of \th-regular types follows as an
easy corollary of the definition of \th-regularity and the results
we have so far in this section.

\begin{corollary}\label{weight4}
\label{cor:reg1}
  A \th-regular type of finite \th-weight has \th-weight 1.
\end{corollary}
\begin{proof}
  Suppose not, and let $p \in \tS(A)$ be a \th-regular type of \th-weight at least 2.
  Without loss of generality
  (since a non-\th-forking extension of a \th-regular type
  is \th-regular), there exists a witness $a, \set{b_1,b_2}$ for $\thpw(p)\ge 2$.
  Moreover, by Lemma \ref{weight1} we may assume that
  $\thw(\tp(b_1/A)) = \thw(\tp(b_2/A)) = 1$.

  Let $a'$ be such that $\tp(a'/Ab_1) = \tp(a/Ab_1)$, $a' \thind_{Ab_1}b_2$. Then
  clearly $a' \thind_A b_2$ (as $b_1, b_2$ are independent over $A$).

  Now notice:

  \begin{itemize}
  \item
    $a \thind_{Ab_2} a'$:
  The type $p$ is \th-regular, so $\tp(a/Ab_2)$ and $\tp(a'/Ab_2)$ are
    weakly \th-orthogonal.
  \item
    $b_1 \nthind_{Ab_2} a$: We know $b_1 \thind_A b_2$ and $b_1 \nthind_A a$.
  \item
    $b_1 \nthind_{Ab_2} a'$: This follows from $a' \equiv_{Ab_1} a$
    (so $a' \nthind_A b_1$), and $b_1 \thind_A b_2$.
  \end{itemize}

  So we have a witness for $\thw(\tp(b_1/Ab_2) \ge 2$, but this type is a non-\th-forking extension
  of $\tp(b_1/A)$, a contradiction.
\end{proof}


We will conclude by pointing out the following unsurprising but
important property of a regular type:

\begin{observation}\label{pregeometry}
    Let $p \in \tS(A)$ be a \th-regular type.
    Define (as usual) for a tuple \c\, of realizations of $p$
    $$\cl_p(\c) = \set{a \models p \colon a \nthind_A \c}$$
    Then $(p^\fC,\cl_p)$ is a pregeometry.
\end{observation}
\begin{proof}
    The proof is quite easy and it is the same as the standard
    proof of the analogue result for (forking) regular types.
\end{proof}

%

\begin{remark}
We should mention that the converse of Observation \ref{pregeometry}
is true assuming stability of $p(x)$ (see \cite{LaPo} for a
definition). In the general --rosy-- context, however, we have been
unable to either prove it or show a counterexample.
\end{remark}


\section{Super-rosy theories and types of finite $\uth$-rank}\label{superrosy}\label{finite uth
rank}\label{finite}

\subsection{Exchange and decomposition in types of finite weight}

The goal of this section is proving that under reasonable
assumptions, any type can be ``decomposed'' into a finite product of
``geometric'' types. Recall that in Theorem \ref{thm:finiteweight}
we in particular proved the following.

\begin{theorem}
\label{theorem:decomposition}
  Let $p \in \tS(A)$ be such that $\thw(p) = n$. Then there exists a set
  $B$, $A \subseteq B$, and $b_1, \ldots, b_n$ \th-independent over $B$
  such that $p \Join \tp(b_1\ldots b_n/B)$ and $\thw(b_i/B) = 1$.
\end{theorem}

We will improve this statement by replacing \th-weight-1 types in
the conclusion by regular (in the super-rosy context) and
\th-minimal (in the finite rank context) types.

\begin{lemma}[Exchange Lemma]
\label{lemma:exchange} Let $a, b_1, \dots, b_n$ be a \th-weight 1
witness of $\thw(\tp(a/A))=n$. Let $q$ be a type with
$\dom(q)\supset A$ such that $q$ is not \th-orthogonal to
$tp(b_n/A)$ and $\thw(q)=1$. Then there is some $b\models q$ and
some $B$ such that $a,B,\langle b_1,\dots, b_{n-1}, b\rangle$
witness $tp(a/A)$ has \th-weight $n$.

Moreover, if $a \bowtie^\tho_A b_1 \ldots b_n$, then we can find $B$
such that both $a \bowtie^\tho_B b_1\ldots b_{n-1}b_n$ and $a
\bowtie^\tho_B b_1\ldots b_{n-1}b$.
\end{lemma}

\begin{proof}
Let $b'\models q$, $B'$ be such that $b_n\thind_{A} B'$, $b'
\thind_A B'$ and $b'\nthind_{B'} b_n$ (such $b'$ and $B'$ exist as
$tp(b_n/A)$ and $q$ are not \th-orthogonal).


Without loss of generality $Bb \thind_{Ab_n} ab_1\ldots b_{n-1}$. In
particular, $ab_1 \ldots b_n \thind_A B$ and $b_1\ldots b_{n-1}
\thind_B b_nb$, and so the set \set{b, b_1, \ldots, b_{n-1}} is
independent over $B$.


Now if $a\thind_B b$, then $a,b$ witness $\thw(tp(b_n/B))\geq 2$
which contradicts our assumptions (via Lemma \ref{weight}). So
$a\nthind_B b$ and by the definition $a,B,\langle b_1,\dots,
b_{n-1}, b\rangle$ witnesses $tp(a/A)$ has \th-weight $n$.

For the ``moreover'' part, assume that $a \bowtie^\tho_A b_1\ldots
b_n$. Recall that by Observations \ref{ab-dom-a} and
\ref{obs:domextend} we may assume $ab_1\ldots b_n \lhd^\tho_A a$
(that is, first replace $A$ with some $A'$ such that $a \thind_AA'$,
$b_1\ldots b_n \thind_AA'$, and $ab_1\ldots b_n \lhd^\tho_{A'} a$,
and then find $B$), hence (by \ref{obs:domextend} again) $ab_1\ldots
b_n \lhd^\tho_B a$. By Lemma \ref{bc-dom-a}, $b_1 \ldots b_{n-1}b
\lhd^\tho_B a$. Finally by Observation \ref{weak-domequivalence} we
have $b_1 \ldots b_{n-1}b_n \bowtie^\tho_B a$ and $b_1 \ldots
b_{n-1}b \bowtie^\tho_B a$, as required.
\end{proof}

%
%
%
%
%
%

\subsection{\th-regularity and decomposition in
the super-rosy case.}

As in the super-stable case, we first prove the existence of
``many'' \th-regular types in a super-rosy theory, which makes the
theory of \th-regular types relevant. We will also point out that
all super-rosy types in a rosy theory have finite \th-weight (hence
the results of the previous section apply in the super-rosy
context).

\begin{proposition}
\label{prp:regexist} Let $T$ be super-rosy. Then every type $p$
with domain $A$ is non-\th-orthogonal to a \th-regular type $q$
with domain $B\supset A$.
\end{proposition}

\begin{proof}
The proof is a variation of the proof of Proposition 5.1.11 in
\cite{wagnerbook}.

Let $\mathcal{P}$ be the set of types $r$ such that
$dom(r)=B\supset A$ and $r$ is not weakly \th-orthogonal to $p$
and let $q$ be a type in $\mathcal{P}$ of minimal $\uth$-rank. Let
$a',b$ be realizations of $p,q$ respectively such that $a'\thind_A
B$ and $a'\nthind_{B} b$.

Suppose $q$ is not \th-regular so there is some $c', b', C'$ such
that $c',b'\models q$, $b'\thind_B C'$, $c'\nthind_B C'$ and
$b'\nthind_{C'} c'$.

Since $tp(b'/B)=tp(b/B)=q$ there is an automorphism fixing $B$ and
sending $c',b',C'$ to elements $c,b,C$ and let $a\models p$
realize a non-\th-forking extension of $tp(a'/bB)$ to $bBC$. So
$a\thind_{bB} C$ and, since $b\thind_B C$, we have by transitivity
that $ab\thind_B C$ which implies that $a\thind_B C$; it follows
that $a\thind_A C$ (recall that $a' \thind_A B$ and $\tp(a/B) =
\tp(a'/B)$).

Notice also that $a\nthind_C b$ (as $a\nthind_B b$ and $a \thind_B
C$).

\medskip

So we have $a\thind_A C$, $a\nthind_C b$, $b\thind_B C$,
$c\nthind_B C$ and $b\nthind_C c$. In particular
\[
\uth(tp(c/C))<\uth(tp(c/B))=\uth(tp(b/B))
\]
and
\[
\uth(tp(b/Cc))<\uth(tp(b/C))=\uth(tp(b/B));
\]
by minimality of $\uth(tp(b/B))$ (among all types in
$\mathcal{P}$) we have that $tp(c/C)$ and $tp(b/Cc)$ are not in
$\mathcal{P}$; so in particular $a\thind _C c$ and $a\thind _{Cc}
b$. By transitivity $a\thind_C bc$ and $a\thind_C b$ a
contradiction.
\end{proof}

\begin{proposition}\label{pro:superrosyfiniteweight}
Let $p(x)$ be a type such that
\[
\uth(p)=\sum_{i=1}^k \omega^{\alpha_i}n_i.
\]
Then $p$ has \th-weight at most $\sum_{i=1}^k n_i$.
\end{proposition}

\begin{proof}
This is word by word the same proof as Theorem 5.2.5 in
\cite{wagnerbook} using the \th-forking version of Lascar's
inequalities (Fact \ref{lascar}).
\end{proof}

As an easy corollary we obtain the following theorem which
strengthens Theorem \ref{theorem:decomposition} in the super-rosy
context.

\begin{theorem}\label{theorem:decomposition2}
 The following hold.

\begin{itemize}
\item Any super-rosy type has finite \th-weight.

\item Let $T$ be super-rosy,
  $p \in \tS(A)$. Then $\thw(p) = n$ for some $n<\om$ and there exists a set
  $B$, $A \subseteq B$, and $b_1, \ldots, b_n$ \th-independent over $B$
  such that $p \Join \tp(b_1\ldots b_n/B)$ and $\tp(b_i/B)$ are \th-regular.
\end{itemize}
\end{theorem}
\begin{proof}
The first item follows immediately from Proposition
\ref{pro:superrosyfiniteweight}.

To prove the second item, notice first that $\thw(p)$ is finite by
Proposition
    \ref{pro:superrosyfiniteweight}.
    Now apply Theorem \ref{theorem:decomposition} combined with
  existence of \th-regular types (Proposition \ref{prp:regexist}) and the Exchange Lemma
  (Lemma \ref{lemma:exchange}), recalling that by
  Corollary \ref{cor:reg1} \th-regular types have \th-weight 1.
\end{proof}

\subsection{Types of finite $\uth$-rank.}

The following proposition is a remarkably interesting result with
many consequences in theories of finite $\uth$-rank.

\begin{proposition}\label{2} The following hold.
\begin{itemize}
\item Let $p(x)=tp(b/A)$ be any type such that $\uth(p)=\alpha+1$.
Then there are tuples $a,e$ such that $\uth(tp(b/Aa))=\alpha$,
$b\thind_A e$, $b\thind_{Aa} e$, $tp(b/Aa)$ strongly divides over
$Ae$ and $\uth(tp(a/Ae))=1$.

\item If $p(x)=tp(b/A)$ is any type of \th-rank $\alpha+1$ then
there is a non-\th-forking extension $tp(b/Ae)$ of $p$ and a tuple
$a\in \acl(Abe)$ such that $tp(a/Ae)$ is minimal.
\end{itemize}
\end{proposition}

\begin{proof} The second item follows immediately from the first one.
To prove the first item, notice that we can choose $a$ so that
$p(x,a)$ is a \th-dividing extension of $p(x)$ and
$\uth(tp(b/Aa))=\alpha$. By Observation \ref{thorn
dividing}\ref{1} there is some $e$ such that $b\thind_{Aa} e$ and
$tp(b/Aae)$ strongly divides over $Ae$ so $a\in acl(Abe)$. Note
that
$$\al=\uth(tp(b/Aa)) = \uth(tp(b/Aae)) < \uth(tp(b/Ae)) \le \uth(tp(b/A)) = \al+1$$

\noindent hence $\uth(tp(b/Ae)) = \uth(tp(b/A)) = \al+1$; in
particular,  $b\thind_A e$.  By Lascar's inequalities we know that
$$\uth(tp(ba/Ae))=\uth(tp(b/Ae)) + \uth(tp(a/Abe))=\alpha+1+0 =
\alpha+1$$

\noindent and
\[
\uth(b/Aae)+ \uth(tp(a/Ae))\leq \uth(tp(ba/Ae))\leq
\uth(b/Aae)\oplus \uth(tp(a/Ae)).
\]

\noindent So
\[
\alpha+ \uth(tp(a/Ae))\leq \alpha +1 \leq \alpha \oplus
\uth(tp(a/Ae)).
\]
and the result follows.\end{proof}

Notice that Proposition \ref{2} provides the inductive step, in
theories of finite $\uth$-rank, for any property which is closed
under non-\th-forking restrictions and coordinatized types (in the
sense that if a type $p$ is coordinatized by types having the
property, then $p$ must have the property). This has nice
consequences (it was strongly used, for example, in \cite{HaOn2}).
Some of the more direct consequences include the following.

\begin{corollary}\label{nice extensions 2}
Let $p(x)$ be any type of finite $\uth$-rank. Then $p(x)$ is
non-\th-orthogonal to a \th-minimal type.
\end{corollary}

\begin{proof}
By Proposition \ref{2} given $p(x)=tp(b/A)$ of finite $\uth$-rank,
there is a non-\th-forking extension $tp(b/Ae)$ and an element
$a\in \acl(Abe)$ such that $tp(a/Ae)$ is \th-minimal. Clearly
$tp(b/Ae)$ and $tp(a/Ae)$ are non-\th-weakly orthogonal.
\end{proof}

\begin{corollary}
Let $p \in \tS(A)$ be a type of finite $\uth$ rank. Then $\thw(p)
= n$ for some finite $n$ and there is a set $B$, $A \subseteq B$,
and $b_1, \ldots, b_n$ independent over $B$ such that $p \Join
\tp(b_1\ldots b_n/B)$ and $\uth(tp(b_i/B)) = 1$.
\end{corollary}

\begin{proof}
The fact that a type of finite $\uth$-rank has finite \th-weight
follows from Proposition \ref{pro:superrosyfiniteweight}. The rest
of the assertions follow from Theorem \ref{theorem:decomposition}
using the Exchange Lemma and Corollary \ref{nice extensions 2}.
\end{proof}

\bigskip

We will conclude this section by making some remarks about
Proposition \ref{2}.

At first glance, it would appear that one could coordinatize a
non-\th-forking extension of any type of finite $\uth$-rank by
repeatedly applying the Proposition. However, this would prove a
coordinatization theorem in the stable case, which is known not to
be true, as the following example shows.

\begin{example}
Let $\mathcal{L}:=\{L, E\}$ be such that $L$ is a ternary relation
and $E$ a binary relation and let $T$ be the theory that states that
$E$ is an equivalence relation with infinitely many infinite classes
and such that $L$ defines an affine space on each $E$-class (so
$L(x,y,z)\Rightarrow E(x,y)\wedge E(x,z)\wedge E(y,z)$). A natural
model $M$ of this theory is a sheaf of affine planes indexed by a
line, where $E(x,y)$ if and only if $x$ and $y$ are in the same
plane and $L(x,y,z)$ happens whenever $x,y,z$ are collinear points
in the same $E$-class.

Let $g$ be $\emptyset$-generic $E$-class in $M$ and $a$ and
$g$-generic point in $g$. The conclusion of Proposition \ref{2}
applied to the type $tp(a/g)$ can be seen in the following way: Let
$b$ be any point in $g$ such that $a\thind_g b$ and let $l$ be the
line through $a$ and $b$ (so that $l:=\{x\mid L(x,a,b)\}$). Then
$tp(a/gb)$ is a non-forking extension of $tp(a/g)$, $L\in \acl(ab)$
and $tp(l/ge)$ is a \th-minimal type.

Going back to coordinatization, if we try to coordinatize
$tp(a/\emptyset)$ the first step is $tp(a/g)$ and $tp(g/\emptyset)$.
The next step, however, would be to coordinatize the non-\th-forking
extension $tp(a/gb)$ of $tp(a/g)$. But $b\nthind_\emptyset a$ (and
the reader can check that this is true for every possible $b$ we can
choose) so this does not help at all in trying to coordinatize
$tp(a/\emptyset)$, nor any non-\th-forking extension of it. In fact,
it is not hard to check that $tp(a/\emptyset)$ cannot be
coordinatized in terms of \th-minimal types.
\end{example}

The example above shows a stable (even super-stable) example where
no coordinatization is possible, and it shows the limitations of
Proposition \ref{2} to get a full coordinatization for super-rosy
theories. The main issue there is that we have no control over the
parameter $e$ we need to get from \th-dividing to strong dividing.
In the affine space, for example, this $e$ cannot be overlooked
nor can we have any control as to where it comes from.

This has two main consequences. On the one hand, once we try to
use Proposition \ref{2} inductively and coordinatize $tp(b/Aae)$
then the $f$ we need can be taken to be such that $b\thind_{Aa} f$
but there is no hope that we can find it such that $b\thind_A f$.
The second consequence is that we can only coordinatize a
non-\th-forking extension of types of rank $\alpha +i$ in types of
rank $\alpha$ and rank $i$ when $i=1$, but we cannot do the same
for $i>1$ without further assumptions.

It seems that this lack of control over the choice of $e$ could be
somewhat solved if we had extra assumptions (definable choice
seems to be the right notion), but even this assumption seems to
not be enough to get any coordinatization-like result beyond
possibly the finite $\uth$-rank case. However, coordinatization is
such a useful tool, and the connections with definable choice are
so unclear, that even results assuming finite $\uth$-rank would be
quite interesting.

\section{Indiscernibles in dependent
theories, strong dependence and weight.}\label{constructing}

Theorem \ref{finite weight} states that if a rosy theory is
strongly dependent then every type has rudimentary finite (and
hence finite) \th-weight. It is natural to ask whether an
analogous notion of weight exists in a general setting (for
example, an arbitrary dependent theory). It has been established
that non-forking plays an important role in the study of dependent
theories. One might wonder, therefore, whether there a notion of
weight based on non-forking which behaves well in dependent
theories. One desired property of such a notion would be: $T$ is
strongly dependent if and only if every type has rudimentarily
finite (and possibly finite) weight.

A possible notion of weight satisfying the property mentioned above
was studied by the authors in \cite{OnUs}. One drawback of that
notion is that it ``measures'' weight of a type with respect to
Morley sequences (and not elements). Although by \cite{Us2} we know
that a Morley sequence is precisely what is needed in a dependent
theory in order to determine a global invariant type (so the
definitions in \cite{OnUs} are quite natural), we were (and still
are) curious whether the definition of weight using Morley sequences
is equivalent to the classical notion.

The answers to these questions are still unclear and they have
motivated further research, such as \cite{Us2}, \cite{KaUs}. We have
discovered that in order to make sense of a notion of weight based
on non-forking, one needs to understand under which conditions there
exist mutually indiscernible sequences starting with given elements
(and to what extent one can ``determine'' the types of those
sequences). Let us explain more precisely what we mean.

Suppose one defined ``weight'' as usual (like in stable theories;
that is, take the definitions in section 2 and replace \th-forking
by forking). Recall that one ingredient of the proof of Theorem
\ref{finite weight} was showing that given a \th-independent set of
elements (tuples), there exist mutually indiscernible sequences
starting with those elements. A natural question whether an
analogous result holds for non-forking is answered positively by
Theorem \ref{mutually ind} below. This is, unfortunately, not enough
in order to prove a result similar to Theorem \ref{finite weight}:
since we do not have any control over the types of those Morley
sequences, it is not clear why they should exemplify dividing
(recall that in the case of \th-forking life was easier, as we could
work with strong dividing, which is exemplified by any infinite
indiscernibly sequence in the type). Of course, if $T$ were stable
(or even simple), there would be no problem, since any Morley
sequence would exemplify dividing.

The discussion above leads to the following two questions:

\begin{qst}\label{qst1}
    To which extent can we ``control the types'' of the
    mutually indiscernible sequences constructed in Theorem
    \ref{mutually ind}? More precisely: what must we assume about
    the set $\lset{b}{i}{\al}$ such that for every indiscernible
    sequences $I_i$ starting with $b_i$ respectively, there are
    indiscernible sequences $I'_i \equiv_{b_i} I_i$ such that $I'_i$
    is indiscernible over $I'_{\neq i}$?
\end{qst}

\begin{qst}\label{qst2}
    To which extent do Morley sequences exemplify dividing
    in a dependent theory?
\end{qst}

It was shown in \cite{Us} that if the types realized by the $b_i$'s
in Question \ref{qst1} are generically stable, then it is enough to
assume non-forking independence. This was the main ingredient in the
proof of the main theorem of Section 8 there: in a strongly
dependent theory, every type has a rudimentarily finite generically
stable weight (below we give a much easier proof of this result
based on Theorem \ref{dividing with Morley} - see Theorem
\ref{theorem:finite gen stable weight}). We could not establish (and
in fact, it is still open) whether assuming non-forking independence
is enough given arbitrary $b_i$ with or without assuming the theory
is dependent, but some progress in this direction has been made, and
the results appear in the second half of this section. It has become
clear in subsequent works (\cite{Us2}, \cite{KaUs}) that Question
\ref{qst1} is related to so-called ``strict non-forking'' defined by
Shelah in \cite{Sh783}.

Concerning Question \ref{qst2}, we prove that, although it is not
the case that \emph{every} Morley sequence exemplifies dividing,
there normally \emph{are} such sequences. This fact has several
consequences, some of which we investigate.

\bigskip In this section we are going to assume that \emph{$T$ is
dependent} unless said otherwise.

\medskip
We will work with classical notions of dividing, forking and
splitting. We assume that the reader is familiar with all of these.
Recall that $a \ind_AB$ stands for ``$\tp(a/AB)$ does not fork over
$A$''.

.

\subsection{Existence of Morley and mutually indiscernible sequences.}

Let us start with the following easy lemma.

\begin{lemma}\label{lem:ind exist} (No need of dependence).
$\left.\right.$
\begin{enumerate}
\item
    Assume  $A \subseteq C$ and $a\ind_AC$ (that is, $\tp(a/C)$ does not fork over $A$).
    Then there exists an $C$-indiscernible sequence
    $I = \lseq{a}{i}{\om}$ with $a_0 = a$. Such $I$ can be chosen to be a
    Morley sequence in $\tp(a/C)$ based on $A$.
\item
    If in addition $C \subseteq D$, then there is $D' \equiv_CD$
    such that $I$ is a Morley sequence in $\tp(a/D')$ based on $A$.
\end{enumerate}
\end{lemma}
\begin{proof}
\begin{enumerate}
\item
    Let $\mu$ be ``big enough'' (that is, so that Fact \ref{Erdos Rado} can be applied for $\lam = |C|+|T|$). Using existence of
    non-forking extensions, we can construct
    a sequence $I' = \lseq{a'}{i}{\mu}$ in $\tp(a/AC)$ based on $A$ such that
    \begin{itemize}
    \item
        $a'_i \equiv_{Ca_{<i}}a_j$ for every $j>i$
    \item
        $a'_i \ind_ACa_{<i}$;
    \end{itemize}
    (note
    that if e.g. $T$ is dependent and $A = \bdd(A)$, we are done,
    since this is also a non-splitting sequence, hence
    indiscernible.)

    By Fact \ref{Erdos Rado} there is an \om-sequence $I =
    \lseq{a}{i}{\om}$ indiscernible over $C$ such that every
    $n$-type of $I$ over $C$ ``appears'' in $I'$. In particular,
    this sequence is still based on $A$ because forking
    $a_i \nind_A Ca_{<i}$ is a property of the type $\tp(a_{\le
    i}/C) $. Since $a_0 \equiv_Ca'_i$ for some $i<\mu$ and
    $a'_i\equiv_Ca'_0 = a$, there is $\sigma \in \Aut(\fC/C)$ taking
    $a_0$ to $a$; by replacing $I$ with the image of $I$ under $\sigma$, (which is still a
    Morley sequence over $C$ based on $A$) we may assume $a_0 =
    a$.

\item
    Since $a \ind_AC$, by existence of non-forking extensions there
    exists $a' \equiv_C a$ such that $a' \ind_AD$. So there is an
    automorphism over $C$ taking $a'$ to $a$ and $D$ to $D'$; now
    apply clause (i).
\end{enumerate}
\end{proof}

Although most properties of non-forking identifying it as an
independence relation is stable or simple theories are generally
false in our contexts, some things can still be said. We will refer
to the fact below as ``transitivity on the left''.

\begin{fact}\label{left forking transitivity}
Let $A, B$ be sets and assume that $I = \lseq{a}{i}{\lam}$ is a
non-forking sequence based on A (that is, $a_i \ind_ABa_{<i}$ for
all $i<\lam$). Then $I\ind_A B$, that is, $\tp(I/AB)$ does not fork
over $A$.
\end{fact}

\begin{proof}
This is Claim 5.16 in \cite{Sh783}.
\end{proof}

\begin{corollary}\label{cor:indep set}
    Let $\set{A_i\colon i<\lam}$ be a non-forking (independent) set over $A$,
    that is, $A_i \ind_A A_{\neq i}$ for all $i$. Then for every
    $W,U \subseteq \lam$ disjoint we have $A_{\in W}\ind_AA_{\in
    U}$.
\end{corollary}
\begin{proof}
    Monotonicity and transitivity on the left.
\end{proof}

\begin{observation}\label{indisc preserve}
    Suppose $I$ is an indiscernible sequence over $A$ and
    $B \ind_A I$. Then $I$ is indiscernible over $AB$.
\end{observation}
\begin{proof}
    By Fact \ref{fact:splitfork}  $\tp(B/AI)$ does not split strongly
    over $A$. Recall that this implies that for every $\a_1,\a_2 \in
    I$ which are on the same $A$-indiscernible sequence we have
    $B\a_1 \equiv_A B\a_2$, which is precisely what we want.
\end{proof}

We proceed to the main results of this section. The first theorem
allows us to construct mutually indiscernible (Morley) sequences
when started with a non-forking sequence.

\begin{theorem}\label{mutually ind}
    Let $T$ be a dependent theory, $A$ a set, and let $\set{a_i\colon i<\ka}$ be a set
    of tuples satisfying
    $a_i \ind_A a_{<i}$.
    Then there are mutually $A$-indiscernible infinite sequences
    $I_i$ (that is, $I_i$ is indiscernible over $A\cup\bigcup\set{I_j\colon j\neq i}$), each $I_i$
    starts with $a_i^0=a_i$.
    Moreover, if $\ka = k$ is finite, then $I_i\ind_AI_{<i}$ for all
    $i$ and
    for $i>0$ we have that
    $I_i$ is a Morley sequence in $\tp(a_i/Aa_{<i})$ based on
    $A$, and if $\tp(a_0/A)$ does not fork over $A$, then we can get
    $I_0$ to be a Morley sequence in $\tp(a_0/A)$ over $A$.
\end{theorem}
\begin{proof}
    Note that by compactness it is enough to prove the theorem when $\ka=k<\om$;
    we will
    prove this by induction on $k$. Clearly, there is nothing to prove
    for $k=1$.

    So assume $\lseq{a}{i}{k+1}$ are given, $a_i \ind_A a_{<i}$. By
    the induction hypothesis there are \lseq{I'}{i}{k} mutually
    indiscernible, $I'_i$ starts with $a_i$, $I'_i$ is a Morley
    sequence over $I'_{<i}$ based on $A$.

    By Lemma \ref{lem:ind exist} (ii) with $D = A\cup\bigcup\lseq{I'}{i}{k} $
    and $C = Aa_{<k}$, there are \lseq{I}{i}{k+1}
    satisfying
    \begin{itemize}
    \item
        $\lseq{I}{i}{k} \equiv_{Aa_{<k}} \lseq{I'}{i}{k}$. So in
        particular these are mutually $A$-indiscernible sequences
        starting with $a_i$; all the non-forking requirements are
        preserved too.
    \item
        $I_k$ is a Morley sequence in $\tp(a_k/AI_{<k})$ based on
        $A$ starting with $a_k$. So in particular it is
        indiscernible over $AI_{<k}$.
    \end{itemize}

    Now we are going to use that $T$ is dependent.
    By Fact
    \ref{left forking transitivity} we know that $I_k \ind_AI_{<k}$, so in
    particular $I_k \ind_{AI_{<i}I_{i \in (i,k)}}I_i$ for all $i<k$.
    By the induction hypothesis $I_i$ is indiscernible over the base
    $AI_{<i}I_{\in(i,k)}$.
    By Observation
    \ref{indisc preserve} this implies that $I_i$ is indiscernible over
    $AI_{<i}I_{>i}$, as required. \end{proof}

Although we find the theorem above interesting on its own, it will
normally not be enough for our applications, since we will often
be interested in starting with given indiscernible sequences (e.g.
exemplifying dividing) and ``make'' them mutually indiscernible,
that is, find mutually indiscernible sequences of the same type
keeping a part of the original configuration (e.g. the first
elements). We make several steps in that direction.

\begin{remark}
    The reader should be aware that related results can be found in Shelah \cite{Sh783} (e.g. Claim
    5.13). However, we believe that Claim 5.13(1) is wrong as stated
    there and the assumptions of Claim 5.13(2) are too strong for
    what we are hoping for, so we prefer not to rely on Shelah's
    work here.
\end{remark}

\begin{lemma}\label{lem:array}
    Let $I = \lseq{a}{i}{\om}$ be an indiscernible sequence over $Ab$ such that $\tp(I/Ab)$ does not fork over $A$.
    Then there exists a Morley sequence \lseq{I}{\al}{\om} over
    $Ab$ based on $A$
    with $I_0 = I$ such that for every $\al$ we have $I_\al$ is
    indiscernible over $AbI_{\neq \al}$ (here $b$ can be an empty, finite or infinite
    tuple).

\end{lemma}
\begin{proof}
    Let $\mu$ be a cardinal.
    We construct by induction on $\al<\mu$ a sequence $I_\al$ such
    that
    \begin{itemize}
    \item
        $I_0 = I$
    \item
        $I_{\al} \equiv_{Ab} I$
    \item
        $I_\al\ind_{A}bI_{<\al}$
    \item
        $I_i$ is indiscernible over $AbI_{\neq i}$ for all $i\le\al$.
    \end{itemize}

    For $\al = 0$ there is nothing to do (note that we are using $I \ind_AAb$).
    Assume that we have $I_{<\al}$ as above. Let $I'_\al$ be such that
    \begin{itemize}
    \item
        $I'_\al \equiv_{Ab} I$
    \item
        $I'_\al \ind_A I_{<\al}b$
    \end{itemize}

    Clearly, we may assume that $I'_\al$ is as long as we wish, hence
    by Fact \ref{Erdos Rado} there exists an $\om$-sequence $I_\al$
    which is indiscernible over $AI_{<\al}b$ and every $n$-type of
    $I_\al$ over $AI_{<\al}b$ ``appears'' in $I'_\al$. Clearly $I_\al \equiv_{Ab} I$.
    By finite character of forking $I_\al \ind_AbI_{<\al}$. By
    monotonicity, for every $\be<\al$ we have $I_\al \ind_{AbI_{<\al, \neq
    \be}} I_{\be}$.  Since $I_{\be}$ is indiscernible over $AbI_{<\al, \neq \be}$,
    by Observation \ref{indisc preserve}, we have $I_\be$ is
    indiscernible over $AbI_{\le \al, \neq \be}$, as required.

    So we obtain a sequence \lseq{I}{\al}{\mu} which is a non-forking
    sequence over $Ab$ (based on $A$) of mutually indiscernible sequences over $A$,
    starting with $I_0$.
    Choosing $\mu$ big enough, using Fact \ref{Erdos Rado} as in
    the proof of Lemma \ref{lem:ind exist}, we may assume in addition that
    $\lseq{I}{\al}{\om}$ is also $Ab$-indiscernible, i.e. a Morley
    sequence over $Ab$ based on $A$.
\end{proof}

\begin{theorem}\label{dividing with Morley}
Any instance of dividing over $A$ which is witnessed by a sequence
$I$ such that $tp(I/A)$ does not fork over $A$, can always be
witnessed by a Morley sequence.

Moreover, if $I:=\langle a_i\rangle_{i\in \omega}$ is indiscernible
over $Ab$ and such that $\{\ph(x,a_i)\}_{i\in \omega}$ is
$k$-inconsistent and $tp(I/Ab)$ does not fork over $A$, then there
is a Morley sequence $\langle a^i \rangle_{i\in \omega}$ over $Ab$
based on $A$  such that $a^0=a_0$ and $\{\ph(x,a^i)\}_{i\in \omega}$
is inconsistent.
\end{theorem}

\begin{proof}
We prove the ``moreover'' part.

Assume that $I = \langle a_i\colon i<\om\rangle$ is an
$Ab$-indiscernible sequence such that $\{\ph(x,a_i)\colon i<\om\}$
is $k$-inconsistent for some $k\in \mathbb{N}$. Denote $a = a_0$.

It is clearly enough to find a Morley sequence as in the statement
of the theorem such that $tp(a/Ab)=tp(a^0/Ab)$. So assume towards a
contradiction that given any Morley sequence $\langle a^i\colon
i<\om\rangle$ over $Ab$ based on $A$ with $tp(a/A)=tp(a^0/A)$ we
have that

\[
\bigwedge_{i<\om} \ph(x,a^i)
\]
is consistent.

Let $a_i^0:=a_i$ and let $I^0:=\langle a_i\colon i<\om\rangle$. By
Lemma \ref{lem:array} there is a Morley sequence of sequences
$\langle I^j\colon j<\om \rangle$ over $Ab$ based on $A$ such that
$I^j$ is indiscernible over $AbI^{\neq j}$; let $I^j:=\langle
a_i^j\colon i<\om\rangle$.

\begin{claim}
Let $\eta:\omega \rightarrow \om$ be a function. Then the sequence
$\langle a_{\eta(i)}^i\colon i<\om \rangle$ is a Morley sequence
over $Ab$ based on $A$.
\end{claim}

\begin{proof}
It is clear from the construction that $\langle a_0^i \rangle$ is a
Morley sequence over $Ab$ based on $A$. But it follows easily by
induction over $n$ (using mutual indiscernibility) that

\[
\tp\left(a^0_{\eta(0)},a^1_{\eta(1)},\dots, a^n_{\eta(n)}/Ab\right)
= \tp\left(a^0_0,a^1_0,\dots, a^n_0/Ab\right)
\]
\end{proof}

We will now prove that given any function $\eta:\omega \rightarrow
\om$ the type

\[
\left\{\left(\ph\left(x,a^i_j\right)\right)^{j=\eta(i)}\colon i,j
<\om\right\}
\]
is consistent, thus contradicting Observation \ref{obs:depstrong}.

The proof will again be by induction. Let

\[
p^n(x):=\bigwedge_{i=0}^n\bigwedge_{j<\om}
\left(\ph\left(x,a^i_j\right)\right)^{j=\eta(i)}\wedge
\bigwedge_{i={n+1}}^\omega \ph\left(x,a^i_{\eta(i)}\right);
\]

Since $\langle a^i_{\eta(i)}\colon i<\om\rangle$ is a Morley
sequence for all $\eta$, our hypothesis implies that
$p^{-1}(x):=\bigwedge_{i={n+1}}^\omega
\ph\left(x,a^i_{\eta(i)}\right)$ is consistent.

Assume that $p^n$ is consistent. Since $I^{n+1}$ is indiscernible
over $A^{n+1} = AI^{\neq n+1}$, and since $I = I^0$ witnesses that
$\ph(x,a_0)$ divides over $A$, it follows that $\ph(x,a_0^{n+1})$
divides over $A^{n+1}$ witnessed by $I^{n+1}$. By Lemma
\ref{negation of dividing}(i)  with $p^n(x)$ here standing for
$p(x)$ there (and the induction hypothesis), we have
\[
p^{n+1}(x):=\bigwedge_{i=0}^{n+1}\bigwedge_{j<\om}
\left(\ph\left(x,a^i_j\right)\right)^{j=\eta(i)} \wedge
\bigwedge_{i\ge{n+2}} \ph\left(x,a^i_{\eta(i)}\right)
\]
is consistent, which completes the induction.

By compactness
\[
\left\{\left(\ph\left(x,a^i_j\right)\right)^{j=\eta(i)}\colon
i,j<\om\right\}
\]
is consistent which contradicts dependence of $T$, see Observation
\ref{obs:depstrong}.
\end{proof}

\begin{remark}
Even though one can easily construct examples where forking does not
satisfy ``existence'', these are almost always quite artificial. In
most of the theories one works with it is always the case that
$tp(a/A)$ does not fork over $A$. In such cases Theorem
\ref{dividing with Morley} just states that any instance of dividing
can be witnessed with a Morley sequence.
\end{remark}

As a consequence, we obtain the following (quite desirable) property
of generically stable types:

\begin{corollary}\label{cor:stable divides Morley}
    Suppose that $\ph(x,b)$ divides over a set $A$ as exemplified by
    an $A$-indiscernible sequence $I$ with $\tp(I/A)$ does not fork
    over $A$. Assume furthermore that $\tp(b/A)$ is generically
    stable. Then any Morley sequence in $\tp(b/A)$ exemplifies that
    $\ph(x,b)$ divides over $A$.
\end{corollary}
\begin{prf}
    By Theorem \ref{dividing with Morley} we know that there exists
    such a Morley sequence; but by stationarity of the generically
    stable type $\tp(b/A)$ over $\acl(A)$, clearly any Morley
    sequence will work.
\end{prf}

\subsection{Strong dependence and finite weight}

We will finish this paper by pointing out some results that follow
from strong dependence and the results we have proved so far.

Let us first recall the classical concept of weight (we will give
the definition without assuming anything on the theory; of course,
it does not always give rise to a well-behaved notion).

As we mentioned before, some of the partial results we get arise
when we restrict the definition of weight to certain kind of types
(for example, generically stable ones). All of the definitions can
be given and studied with either forking or \th-forking. However,
forking is clearly the right notion for generically stable types
(see \cite{Us}) so from now on we will just work with the standard
classification theory notions (forking, splitting and dividing).

\begin{definition}\label{dfn:weight}
$\left.\right.$

\begin{itemize}

\item Let $p(x)$ be any type over some set $A$. We will say that
$a, \langle b_i\rangle_{i=1}^n$ witnesses $\pwt(p(x))\geq n$
(\emph{forking pre-weight of $p$} is at least $n$) if $a\models
p(x)$, $\seq{b_i}_{i=1}^n$ is $A$-independent and $a\nind_A b_i$ for
all $i,j$. If $n$ is maximal such that such a witness exists, we
will say that $a, \langle b_i\rangle_{i=1}^n$ witnesses
$\pwt(p(x))=n$ and that $p$ \emph{has forking pre-weight $n$}.

\item Let $p(x)$ be any type over some set $A$. We will say that
$a,B, \langle b_i\rangle_{i=1}^n$ witnesses $\wt(p(x))\geq n$
(\emph{forking weight of $p$} is at least $n$) if $a\models p(x)$,
$a\ind_A B$, $\seq{b_i}_{i=1}^n$ is $B$-independent sequence and
$a\nind b_i$ for all $i,j$. If $n$ is maximal such that such a
witness exists, we will say that $a,B, \langle b_i\rangle_{i=1}^n$
witnesses $\wt(p(x))=n$ and that $p$ has \emph{forking weight $n$}.
\end{itemize}
\end{definition}

Similar definitions can be given requiring that the types
$\tp(b_i/A)$ above are generically stable (see Fact
\ref{fact:genstab}), obtaining \emph{generically stable pre-weight
and weight of $p$} denoted $\gstpw(p)$ and $\gstw(p)$ respectively.
See section 8 of \cite{Us} for precise definitions.

\bigskip

The following follows easily from the results we have so far.

\begin{theorem}\label{theorem:finite stable weight}

\label{theorem:finite gen stable weight}
    Assume $T$ is strongly dependent. Then every finitary type over a model (or just over
    an extension base) has rudimentary
    finite generically stable pre-weight.
\end{theorem}

\begin{proof}

This was original proved in section 8 of \cite{Us}; however, having
established Corollary \ref{cor:stable divides Morley}, the proof of
Theorem \ref{theorem:finite gen stable weight} becomes much easier
than the original one given in \cite{Us}. Indeed, just like in a
stable theory, given an instance of pre-weight $k$, as exemplified
by \set{b_i\colon i<k} realizing generically stable types, we can
simply construct ``mutually'' Morley sequences $I_i$ starting with
$b_i$, which by stationarity will be mutually indiscernible. By
Corollary \ref{cor:stable divides Morley} they exemplify dividing,
thus form a dividing pattern.
\end{proof}

We would like to generalize Theorem \ref{theorem:finite gen stable
weight} to forking weight. For notational simplicity, let us
concentrate on randomness patterns of depth 2 (analogous statements
for larger depth will follow by a simple induction).

The following theorem weakens the assumptions of Theorem
\ref{theorem:finite gen stable weight} somewhat, requiring only one
of the types to be generically stable.

\begin{theorem}\label{theorem:finite generic weight}
Let $a,b,c$ be elements and $A$ be a subset of a model $M$ of a
dependent theory $T$. If $\tp(b/A)$ is generically stable, $b\ind_A
c$, $\tp(a/Ab)$ divides over $Ac$, and $\tp(a/Ac)$ divides over
$Ab$, then there is a randomness pattern of depth 2 for $\tp(a/A)$.
In particular, $T$ is not dp-minimal.
\end{theorem}

\begin{proof}
Let $a,b,c,A$ be as in the statement of the theorem. By definition
of dividing and Theorem \ref{dividing with Morley}, for every
cardinal $\mu$ there is a sequence $I:=\lseq{b}{i}{\mu}$ which is
Morley over $Ac$ and a formula $\ph(x,y)$ such that $\models
\ph(a,b_0)$ and $\{ \ph(x,b_i)\}_{i\in \omega}$ is $k$-inconsistent
for some $k\in \mathbb{N}$.

Since $b_0\ind_A c$ and $\tp(b_0/A)=\tp(b_i/A)$ is generically
stable, it follows that $\tp(b_0/Ac)=\tp(b_i/Ac)$ is generically
stable for all $i$. Using transitivity (Corollary 4.13 or Theorem
7.6 in \cite{Us}), it is easy to prove that $\langle b_i
\rangle_{i\in \omega}$ is also a Morley sequence over $A$. So in
particular, by right transitivity

\[
\langle b_i \rangle_{i\in \omega}\ind_A c
\]

By definition of dividing, there is an $Ab_0$-indiscernible sequence
$J:=\lseq{c}{j}{\om}$ and a formula $\psi(x,y)$ such that $c=c_0$,
$\models \psi(a,c_0)$ and $\{ \psi(x,s_i)\}_{i\in \omega}$ is
$k'$-inconsistent for some $k'\in \mathbb{N}$; by taking the
maximum, we may assume that $k=k'$.

By extension, there is some $J'\equiv_{Ac_0} J$ such that $I\ind_A
J'$. We may assume that $I$ is as long as we want, so by Fact
\ref{Erdos Rado} there is an \om-sequence $I' \equiv_{Ac_0} I$
which is indiscernible over $J'$ such that every $n$-type of $I'$
over $AJ'$ ``appears'' in $I$. In particular, it is still the case
that $I' \ind_A J'$. Moreover, since $I$ is indiscernible over
$Ac_0$, we have (denoting $I' = \lseq{b'}{i}{\om}$) $c_0b'_0
\equiv_A c_0b_0 = cb$, in particular, $\exists
x\ph(x,b'_0)\land\psi(x,c_0)$ is consistent with $\tp(a/A)$,
whereas the formulas $\ph(x,b'_0)$ and $\psi(x,c_0)$ divide over
$Ac_0$ and $Ab_0'$ respectively, exemplified by the sequences
$I'$, $J$.

Note that $I'$ is indiscernible over $AJ'$ and by Observation
\ref{indisc preserve} (since $I' \ind_A J'$), $J'$ is. Lemma
\ref{negation of dividing} implies that we obtain a randomness
pattern of depth 2 for $\tp(a/A)$ as required.
\end{proof}

We will finish this section by partial results which do not assume
anything on the types. The following result addresses the question
about the possible assumptions on indiscernible sequences (e.g.,
exemplifying dividing and pre-weight $k$; we concentrate on $k=2$)
are sufficient for achieving results such as constructing a
randomness pattern. They are not strict generalizations of Theorems
\ref{theorem:finite gen stable weight} or \ref{theorem:finite
generic weight} because we need to include requirements on the
sequences, and not just their first elements. However, it has the
advantage of removing the generic stability assumptions completely.

\begin{proposition}\label{prp:mut ind}
    Let $I = \lseq{a}{i}{\om}$, $J = \lseq{b}{j}{\om}$ be
    $A$-indiscernible sequences such that $I \ind_A b_0$ and $b_0
    \ind_A a_0$. Then there exist $I',J'$ mutually $A$-indiscernible
    with $I' \equiv_{Aa_0} I$, $J' \equiv_{Aa_0} J$.
\end{proposition}
\begin{proof}
    Since $b_0 \ind_A a_0$, there is $b''_0 \ind_A I$ satisfying
    $b''_0 \equiv_{Aa_0} b_0$. Applying an automorphism over $Aa_0$
    taking $b''_0$ to $b_0$, we obtain a new sequence $I'$ with the
    same type over $Aa_0$ as $I$; so without loss of generality $I =
    I'$, and in addition to our assumptions we have $b_0 \ind_A I$,
    hence $I$ is $Ab_0$-indiscernible.

    Since $I\ind_Ab_0$, there is $I''\equiv_{Ab_0}I$ satisfying $I''
    \ind_A J$. Applying an automorphism taking the first element of
    $I''$ to $a_0$ (fixing $Ab_0$), we obtain new sequences $I',J'$
    starting with $a_0,b_0$ respectively, satisfying the same type
    over $A$ as $I,J$ respectively. So without loss of generality $I
    = I', J = J'$.

    Prolonging $I$ and applying Fact \ref{Erdos Rado}, we get $I''$
    indiscernible over $AJ$, ``similar'' to $I$ over $AJ$; in
    particular, $I''$ is $Ab_0$ - indiscernible (and has the same type
    over $Ab_0$ as $I$) and $I'' \ind_A J$. Applying an automorphism
    $\sigma$ over $Ab_0$ taking $I''$ onto (an initial segment of) $I$, we
    get, denoting $J' = \sigma(J)$:
    \begin{itemize}
    \item
        $J' \equiv_{Ab_0} J$
    \item
        $I$ is indiscernible over $AJ'$
    \item
        $I \ind_A J'$, hence $J'$ is indiscernible over $AI$
    \end{itemize}
    This finishes the proof.
\end{proof}

The following is easy now:

\begin{corollary}
    Let $p \in \tS(A)$, $I = \lseq{a}{i}{\om}$, $J = \lseq{b}{j}{\om}$ be
    $A$-indiscernible sequences such that $I \ind_A b_0$ (e.g. $I$ is a Morley sequence over $Ab_0$ based on $A$)
    and $J \ind_A a_0$ (or just $b_0
    \ind_A a_0$) such that
    \begin{itemize}
    \item
        $I,J$ exemplify that $\ph(x,a_0)$,
        $\ps(x,b_0)$ divide over $A$
    \item
        $\ph(x,a_0)\land\ps(x,b_0)$ is consistent with $p(x)$
    \end{itemize}
    Then there exists a dividing pattern and a randomness pattern in
    $p$. In particular, $p$ is not dp-minimal.
\end{corollary}

The results above seem to suggest that for discussion of weight in
dependent theories it is not enough to look just at the first
elements of the sequences of a dividing pattern (the usual notion
of forking weight). So let us conclude with the following notion
of \emph{splitting weight} which behaves quite nicely.

\begin{definition}\label{dfn:split_weight}
$\left.\right.$

\begin{itemize}

\item Let $p(x)$ be any type over some set $A$. We will say that
$a, \langle b^0_ib^1_i\rangle_{i=1}^n$ witnesses
$split-\pwt(p(x))\geq n$ (\emph{splitting pre-weight of $p$} is at
least $n$) if $a\models p(x)$, $\seq{b^0_ib^1_i}_{i=1}^n$ is
$A$-independent and $a\nind_A b^0_ib^1_i$ for all $i$ in a very
strong way, namely:
\begin{itemize}
    \item[$\diamondsuit$] There exists a formula $\ph_i(x,y)$ such that
    $$\models \ph_i(a,b^0_i)\land\neg\ph_i(a,b^1_i)$$.
\end{itemize}
 If $n$ is
maximal such that such a witness exists, we will say that $a,
\langle b_i\rangle_{i=1}^n$ witnesses $split-\pwt(p(x))=n$ and that
$p$ \emph{has splitting pre-weight $n$}.

\item Let $p(x)$ be any type over some set $A$. We will say that
$a,B, \langle b^0_ib^1_i\rangle_{i=1}^n$ witnesses
$split-\wt(p(x))\geq n$ (\emph{splitting weight of $p$} is at least
$n$) if $a\models p(x)$, $\seq{b^0_ib^1_i}_{i=1}^n$ is
$A$-independent and
\begin{itemize}
    \item[$\diamondsuit$] There exists a formula $\ph_i(x,y)$ such that
    $$\models \ph_i(a,b^0_i)\land\neg\ph_i(a,b^1_i)$$.
\end{itemize} If $n$ is maximal such that such a
witness exists, we will say that $a,B, \langle b_i\rangle_{i=1}^n$
witnesses $split-\wt(p(x))=n$ and that $p$ has \emph{splitting
weight $n$}.
\end{itemize}
\end{definition}

\begin{observation}
    Let $p \in \tS(A)$. Then $p$ is strongly dependent if and only if its
    splitting pre-weight is rudimentary finite.
\end{observation}
\begin{proof}
    The ``if'' direction is clear. For the ``only if'' direction, using Theorem
    \ref{mutually ind}, one can construct an array of mutually
    indiscernible (Morley) sequences starting with $b^0_ib^1_i$. The
    rest is easy.
\end{proof}

The next observation (whose proof is easy and similar to the
previous one) connects dependence in general to weight:

\begin{observation}
    A theory $T$ is dependent if and only if every type has a
    bounded splitting weight.
\end{observation}

We see that it is unnecessary to assume (as we did in Proposition
\ref{prp:mut ind}) that $I \ind_A b_0$, it is enough to look at the
first two elements of the sequence $I$; but this seems to be
important (in case $b_0$ is not generically stable). In a sense,
what we do is replacing in the dividing/randomness pattern the
formulas $\ph_i(x,y_i)$ with $\ph'_i(x,y^0_iy^1_i) = [\ph_i(x,y^0_i)
\equiv \ph_i(x,y^1_i)]$ and considering a dividing pattern (=
witness for high pre-weight) with respect to these new formulas.

\bigskip

We would like to finish by remarking that results in this section
pretend to be a first approach to characterize strong dependence
by a notion of finite weight. A complete result of this type would
be quite interesting and, we believe, very useful. However, it is
not clear that plain forking is the right notion for this. As we
mentioned before, it seems that strict non-forking is the right
way to go, and we refer the reader to \cite{Us2} and
\cite{KaUs} for more details.

\appendix
\section{cc-forking and cc-dividing}

The following theorem is proven in \cite{OnUs}. We include the proof
here for the sake of completeness.

\begin{theorem}\label{equivalent}
The following are equivalent for any $p\in S_1(A)$.
\begin{enumerate}
\item $p$ admits $n$-$cc$-\th-forking.

\item $p$ admits $n$-$cc$-\th-dividing.

\item There is an extension $p(x,B)$ of $p(x)$ such that $p(x,B)$
admits $n$-$cc$-strong dividing.

\end{enumerate}
\end{theorem}

\begin{proof}
Any witness for $n$-$cc$-strong dividing is a witness of
$n$-$cc$-\th-dividing and any witness for $n$-$cc$-\th-dividing is a
witness of $n$-$cc$-\th-forking. We will prove the other
implications for $n=2$. The general case will follow by a
straightforward induction on $n$ using the properties of \th-forking
in rosy theories.

(i) \then (ii). Let $$\{ \ph_i(x,\a), \psi(x,\b)\}, A$$ be a
$2-cc$-\th-forking witness for $p$. By definition there are finitely
many formulas $\ph_i(x,\a_i')$, $\psi_j(x,\b_j')$ and tuples
$\c',\bar d'$ such that $\ph(x,\a)\vdash \bigvee_{i=1}^{k_a}
\ph_i(x,\a_i')$, $\psi(x,\b)\vdash \bigvee_{j=1}^{k_b}
\psi_j(x,\b_j')$ and $\ph_i(x,\a_i')$ strongly divides over $A\c'$
and $\psi_j(x,\b_j')$ strongly divides over $A\bar d'$.

By hypothesis $\a \thind_A \b$ so by extension of \th-independence
we can find $\a''\models tp(\a/A\b)$ such that $\a''\ind_A
\seq{\b_j'}\bar d'\b$. Let $\seq{\a''_i}, \c''$ be images of
$\seq{\a'_i},\c'$ under an automorphism that fixes $A,\b$ and sends
$\a$ to $\a''$.

Using extension on the other side there are $\b'' \seq{\b''_j} \bar
d''\models tp(\b \seq{\b'_j}\bar d'/A\a'')$ such that
\[\seq{\a_i''}\bar c''\a''\thind_A \b'' \seq{\b_j''}\bar d''.\]

But $tp(\a''\b''/A)=tp(\a''\b/A)=tp(\a\b/A)$ so by applying an
automorphism over $A$, we can find \seq{\a_i}, $\c$, \seq{\b_j},
$\bar d$ such that $\seq{\a_i}\c\a \ind_A \seq{\b_j}\bar d\b$.

So in particular we have $\tp(\seq{\a_i}\c/A\a) =
\tp({\seq{\a'_i}}\c'/A\a)$ and $\tp(\seq{\b_i}\bar d/A\b) =
\tp({\seq{\b'_i}}\bar d'/A\b)$.

Therefore
\begin{equation}\label{equ:1}
  \ph(x,\a)\models \bigvee_{i=1}^{k_a} \ph_i(x,\a_i),\;\;
  \psi(x,\b)\models \bigvee_{i=1}^{k_a} \psi_i(x,\b_i),
\end{equation}
and
\begin{equation}\label{equ:2}
\begin{array}{l}
  \ph(x,\a_i) \; \mbox{strongly divides over} \; A\c \; \mbox{for all} \; i, \\
  \psi(x,\b_j) \; \mbox{strongly divides over} \; A\bar d\; \mbox{for all}\; j
\end{array}
\end{equation}

Since $\ph(x,\a)\wedge \psi(x,\b)$ is consistent with $p$, it is
clear from (\ref{equ:1}) that the conjunction $\ph_i(x,\a_i)\wedge
\psi_j(x,\b_j)$ is consistent with $p$ for some $i,j$. By
monotonicity of \th-forking independence we know that $\a_i\ind_A
\b_j$ so (\ref{equ:2}) implies that $(\ph(x,\a_i), \psi(x,\b_j)),A$
is a witness for $cc$-\th-dividing.

\bigskip

(ii)$\Rightarrow$ (iii). Once again we will prove the case $n=2$.

Let $$\{ \ph_i(x,\a), \psi(x,\b)\}, A$$ be a $2-cc$-\th-dividing
witness for $p$. Let $D'$ and $E'$ be supersets of $A$ such that
$\ph(x,\a)$ strong divides over $D'$ and $\psi(x,\b)$ strong divides
over $E'$. Since $\a\thind_A \b$ we can, by extension (as in the
proof of (i) \then (ii)), find $D,E$ satisfy types $tp(D'/A\a)$ and
$tp(E'/A\b)$ respectively and such that $\a D\thind_A \b E$; so in
particular $\ph(x,\a)$ strong divides over $D$ and $\psi(x,\b)$
strong divides over $E$, $\a\thind_{D} E$, $\b\thind_E D$, and
$\a\thind_{D\cup E} \b$.

Since by definition $\a\not\in acl(D)$ and $\b\not\in \acl(E)$ we
get that $\a,\b\not\in \acl(ED)$: e.g., $\a\not\in\acl(D)$, but
$\a\thind_ED$, so $\a\not\in\acl(ED)$.

So
\begin{equation}\label{equ:3}
\begin{array}{l}
  \ph(x,\a) \; \mbox{strongly divides over} \; E\cup D, \\
  \psi(x,\b) \; \mbox{strongly divides over} \; E\cup D, \\
\a\thind_{D\cup E} \b.
\end{array}
\end{equation}

Let $B:=D\cup E$ and let $p(x,B,\a,\b)$ be a non-\th-forking
extension of $p(x)\cup \{\ph(x,\a)\cup \psi(x,\b)\}$ and let
$p(x,B)$ be the restriction of $p(x,B,\a,\b)$ to $B$. All the
conditions in the definition of $2-cc$-strong dividing are
satisfied which completes the proof of the theorem.
\end{proof}

\end{document}